\documentclass[11pt,reqno,final]{amsart}
\usepackage{amsmath, amssymb, epsfig}
\usepackage{mathrsfs}
\usepackage{color}
\usepackage{alg}
\usepackage{url}

\usepackage[font=small,margin=10pt,labelfont={bf},labelsep={space}]{caption}
\usepackage{subfig}
\usepackage{wrapfig}

\usepackage[scaled]{helvet} 
\usepackage{fourier} 
\usepackage{bm} 

\makeatletter
\def\algspacing{\alg@unmargin}
\makeatother

\usepackage[margin=1in]{geometry}

\newlength{\algorithmwidth}
\theoremstyle{plain}
\newtheorem{theorem}{Theorem}[section]
\newtheorem{proposition}[theorem]{Proposition}
\newtheorem{corollary}[theorem]{Corollary}
\newtheorem{lemma}[theorem]{Lemma}

\theoremstyle{definition}
\newtheorem{definition}[theorem]{Definition}

\theoremstyle{remark}
\newtheorem*{remark}{Remark}

\numberwithin{theorem}{section}
\numberwithin{equation}{section}

\newcommand{\bigO}{\mathrm{O}}

\DeclareMathOperator*{\trace}{trace}

\def \E {\mathbb{E}}

\def \eps {\varepsilon}

\def \rank {{\rm rank }}

\newcommand{\vct}[1]{\bm{#1}}
\newcommand{\mtx}[1]{\bm{#1}}

\newcommand{\ip}[2]{\left\langle {#1}, \, {#2} \right\rangle}
\newcommand{\norm}[1]{\lVert {#1} \rVert}
\newcommand{\normsq}[1]{\norm{#1}^2}
\newcommand{\enorm}[1]{\norm{#1}_2}
\newcommand{\enormsq}[1]{\enorm{#1}^2}
\newcommand{\fnorm}[1]{\norm{#1}_{\rm F}}
\newcommand{\fnormsq}[1]{\fnorm{#1}^2}
\newcommand{\infnorm}[1]{\norm{#1}_{\infty}}
\newcommand{\infnormsq}[1]{\infnorm{#1}^2}
\newcommand{\Expect}{\E}
\newcommand{\abs}[1]{\left\vert {#1} \right\vert}

\newcommand{\pinv}{\dagger}
\newcommand{\cnst}[1]{\mathrm{#1}}
\newcommand{\Id}{\mathbf{I}}
\newcommand{\Prob}[1]{\mathbb{P}\left\{ #1 \right\}}
\newcommand{\diag}{\operatorname{diag}}

\begin{document}

\title[Paved with Good Intentions]{Paved with Good Intentions: \\ Analysis of a Randomized Block Kaczmarz Method}
\author{Deanna Needell and Joel A. Tropp}
\date{17 August 2012}

\begin{abstract}
The block Kaczmarz method is an iterative scheme for solving overdetermined least-squares problems.  At each step, the algorithm projects the current iterate onto the solution space of a subset of the constraints.  This paper describes a block Kaczmarz algorithm that uses a randomized control scheme to choose the subset at each step.  This algorithm is the first block Kaczmarz method with an (expected) linear rate of convergence that can be expressed in terms of the geometric properties of the matrix and its submatrices.  The analysis reveals that the algorithm is most effective when it is given a good \emph{row paving} of the matrix, a partition of the rows into well-conditioned blocks.  The operator theory literature provides detailed information about the existence and construction of good row pavings.  Together, these results yield an efficient block Kaczmarz scheme that applies to many overdetermined least-squares problem.
\end{abstract}

\maketitle

\section{Introduction} 

The Kaczmarz method~\cite{Kac37:Angenaeherte-Aufloesung} is an iterative algorithm for solving overdetermined least-squares problems.  Because of its simplicity and performance, this scheme has found application in fields ranging from image reconstruction to digital signal processing~\cite{SS87:Incorporation-A-Priori,CFMSS92:New-Variants,FS95:Kaczmarz-Based-Approach,Nat01:Mathematics-Computerized}.  At each iteration, the basic Kaczmarz method makes progress by enforcing a single constraint, while the block Kaczmarz method~\cite{Elf80:Block-Iterative-Methods} enforces many constraints at once.  This paper introduces a randomized version of the block Kaczmarz method that converges with an expected linear rate, and we characterize the performance of this algorithm using geometric properties of the blocks of equations.  This analysis leads us to consider the concept of a \emph{row paving} of a matrix, a partition of the rows into well-conditioned blocks.  We summarize the literature on row pavings, and we explain how this theory interacts with the block Kaczmarz method.  Together, these results yield an efficient block Kaczmarz scheme that applies to many overdetermined least-squares problems. 



\subsection{Standing Assumptions and Notation} \label{sec:assumptions}

Let $\mtx{A}$ be a real or complex $n \times d$ matrix with full column rank, and suppose that $\vct{b}$ is a vector with dimension $n$.
Consider the overdetermined least-squares problem
\begin{equation} \label{eqn:basic-ls}
\mathrm{minimize} \quad \enormsq{ \mtx{A} \vct{x} - \vct{b} }.
\end{equation}
The symbol $\norm{\cdot}_p$ refers to the $\ell_p$ vector norm for $p \in [1, \infty]$.
We write $\vct{x}_{\star}$ for the unique minimizer of~\eqref{eqn:basic-ls}, and we introduce the residual vector $\vct{e} := \mtx{A} \vct{x}_{\star} - \vct{b}$.

To streamline our discussion, we will often assume that each \emph{row} $\vct{a}_i$ of the matrix $\mtx{A}$ shares the same $\ell_2$ norm:
\begin{equation} \label{eqn:std-matrix}
\enorm{ \vct{a}_i } = 1
\quad\text{for each $i = 1, \dots, n$.}
\end{equation}
We say that $\mtx{A}$ is \emph{standardized} when~\eqref{eqn:std-matrix} is in force.

The spectral norm is denoted by $\norm{\cdot}$, while $\fnorm{\cdot}$ represents the Frobenius norm.
When applied to an Hermitian matrix, the maps $\lambda_{\min}$ and $\lambda_{\max}$ return the algebraic minimum and maximum eigenvalues.
For a $p \times q$ matrix $\mtx{W}$, we arrange the singular values as follows.
$$
\sigma_{\max}(\mtx{W}) := \sigma_1(\mtx{W}) \geq \sigma_2(\mtx{W})
	\geq \dots \geq \sigma_{\min\{p, q\}}(\mtx{W}) =: \sigma_{\min}(\mtx{W}).
$$
The minimum singular value is positive if and only if $\mtx{WW}^*$ or $\mtx{W}^* \mtx{W}$
is nonsingular.  We define the \emph{condition number} $\kappa(\mtx{W}) := \sigma_{\max}(\mtx{W}) / \sigma_{\min}(\mtx{W})$.  The dagger $\pinv$ denotes the Moore--Penrose pseudoinverse.  When $\mtx{W}$ has full row rank, its pseudoinverse is determined by the formula $\mtx{W}^\pinv := \mtx{W}^* (\mtx{WW}^*)^{-1}$.



\subsection{The Simple Kaczmarz Method} \label{sec:simple-kacz}


The Kaczmarz method is an iterative algorithm that produces an approximation to the minimizer $\vct{x}_{\star}$ of the least-squares problem~\eqref{eqn:basic-ls}.  The method commences with an arbitrary guess $\vct{x}_0$ for the solution.  At the $j$th iteration, we select a row index $t = t(j)$ of the matrix $\mtx{A}$, and we project the current iterate $\vct{x}_{j-1}$ onto the solution space of the equation $\ip{ \vct{a}_t }{ \vct{x} } = b_t$.  That is,
\begin{equation} \label{eqn:simple-kacz}
\vct{x}_j = \vct{x}_{j-1} + \frac{b_t - \ip{ \vct{a}_t }{ \vct{x}_{j-1} }}
{\enormsq{\vct{a}_t}} \, \vct{a}_t.
\end{equation}
This process continues until it triggers an appropriate convergence criterion.


%

To develop a complete algorithm, we also need a control mechanism that specifies how to select rows.  For example, the most classical approach cycles through the rows in order.  Instead, we focus on a modern formulation that uses a \emph{randomized} control mechanism.  Randomization has several benefits: the resulting algorithm is easy to analyze, it is simple to implement, and it is often effective in practice.

Our primary reference is the randomized Kaczmarz algorithm recently proposed by Strohmer and Vershynin~\cite{SV09:Randomized-Kaczmarz}.
When $\mtx{A}$ is standardized, their method operates as follows.  At iteration $j$, independently of all previous random choices, the algorithm draws the row index $t(j)$ uniformly at random from the set $\{1, \dots, n\}$ of all row indices.
Then the current iterate is updated using the rule~\eqref{eqn:simple-kacz}.
The paper~\cite{SV09:Randomized-Kaczmarz} provides a short, elegant proof that
this iteration converges at an expected linear rate to the solution $\vct{x}_{\star}$ of a consistent least-squares problem (i.e., where the residual $\vct{e}$ is zero).


Needell~\cite{Nee10:Randomized-Kaczmarz} has extended the argument of~\cite{SV09:Randomized-Kaczmarz} to the case of an \emph{inconsistent} least-squares problem.
For a standardized matrix $\mtx{A}$, Needell's error estimate reads
\begin{equation}\label{eqn:Needell-rate}
\Expect \enormsq{ \vct{x}_j - \vct{x}_\star }
	\ \leq \ \left[ 1 - \frac{\sigma^2_{\min}(\mtx{A})}{n} \right]^j \enormsq{ \vct{x}_0 - \vct{x}_\star }
	\ + \ \frac{n \infnormsq{\vct{e}}}{\sigma_{\min}^2(\mtx{A})}.
\end{equation}
In words, the randomized Kaczmarz method converges in expectation at a linear rate%
\footnote{Mathematicians often use the term \emph{exponential convergence} for the concept numerical analysts call \emph{linear convergence}.}
until it reaches a fixed ball about the true solution $\vct{x}_{\star}$, at which point the error may cease to decay.  The radius of this ball roughly equals the second term in~\eqref{eqn:Needell-rate}, while
the convergence rate is controlled by the bracket.
When the residual $\vct{e}$ is zero, the bound~\eqref{eqn:Needell-rate} reduces to the error estimate from~\cite{SV09:Randomized-Kaczmarz}.

When $\mtx{A}$ is an $n\times d$ standardized matrix whose columns are well conditioned, the minimum singular value $\sigma^2_{\min}(\mtx{A}) \geq \textrm{const} \cdot n / d$.  In this case, the error bound~\eqref{eqn:Needell-rate} simplifies to
$$
\Expect \enormsq{ \vct{x}_j - \vct{x}_\star }
	\ \lesssim \ \left[ 1 - \frac{\textrm{const}}{d} \right]^j \enormsq{ \vct{x}_0 - \vct{x}_\star }
	\ + \ \frac{d}{\textrm{const}} \infnormsq{\vct{e}}.
$$
It follows that $\bigO(d)$ iterations of the Kaczmarz method suffice to reduce the error by a constant fraction, provided that the squared error is substantially larger than $d \infnormsq{\vct{e}}$.

\subsection{The Block Kaczmarz Method} \label{sec:block-kacz}


In some situations~\cite{EHL81:Iterative-Algorithms}, practitioners prefer to use a block version of the Kaczmarz method to solve the least-squares problem~\eqref{eqn:basic-ls}.  We consider a formulation due to Elfving~\cite{Elf80:Block-Iterative-Methods}.  This procedure begins with an initial guess $\vct{x}_0$ for the solution.  At each iteration $j$, we select a subset $\tau = \tau(j)$ of the row indices of $\mtx{A}$, and we project the current iterate $\vct{x}_{j-1}$ onto the solution space of $\mtx{A}_{\tau} \vct{x} = \vct{b}_{\tau}$, the set of equations listed in $\tau$.  That is,
\begin{equation}\label{eqn:block-kacz}
\vct{x}_j = \vct{x}_{j-1} + (\mtx{A}_\tau)^\pinv(\vct{b}_\tau - \mtx{A}_\tau\vct{x}_{j-1}).
\end{equation}
This process continues until it has converged.
We have written $\mtx{A}_{\tau}$ for the row submatrix of $\mtx{A}$ indexed by $\tau$, while $\vct{b}_{\tau}$ is the subvector of $\vct{b}$ with components listed in $\tau$.  We assume that the row submatrix $\mtx{A}_{\tau}$ is fat%
\footnote{A $p \times q$ matrix is \emph{fat} when $p \leq q$.},
so the pseudoinverse~\eqref{eqn:block-kacz} returns the solution to an underdetermined least-squares problem.




To specify a block Kaczmarz algorithm, one must decide what blocks of indices are permissible, as well as the mechanism for selecting a block at each iteration.  In this paper, we study a version that is based on two design decisions.  First, this algorithm requires a partition $T = \{ \tau_1, \dots, \tau_m \}$ of the row indices of $\mtx{A}$.  The method only considers blocks of indices that appear in the partition $T$.  Second, we use a simple randomized control scheme to choose which block to enforce.  At each iteration, independently of all previous choices, we draw a block $\tau$ uniformly at random from the partition $T$.  These decisions lead to Algorithm~\ref{alg:randomized-block-kaczmarz}.
We postpone a detailed discussion on implementation to Section~\ref{sec:implementation}.

We make no claim that randomized selection provides the optimal sequence for choosing blocks.  As with the simple Kaczmarz method, the scaling of the rows of the matrix can play a significant role in the behavior of the algorithm.  See~\cite{CHJ09:Note-Behavior,SV09:Comments-Randomized} for a discussion of this issue.

\begin{center}
\begin{algorithm}[tb]
\caption{Block Kaczmarz Method with Uniform Random Control}
	\label{alg:randomized-block-kaczmarz}
\begin{center} \fbox{
\begin{minipage}{.95\textwidth} 
\vspace{4pt}
\alginout{\begin{itemize}
\item	Matrix $\mtx{A}$ with dimension $n \times d$
\item	Right-hand side $\vct{b}$ with dimension $n$
\item	Partition $T = \{\tau_1, \dots, \tau_m\}$ of the row indices $\{1, \dots, n\}$
\item	Initial iterate $\vct{x}_0$ with dimension $d$
\item	Convergence tolerance $\varepsilon > 0$
\end{itemize}}
{An estimate $\hat{\vct{x}}$ for the solution to $\min_{\vct{x}} \enormsq{ \mtx{A}\vct{x} - \vct{b} }$
}
\vspace{8pt}\hrule\vspace{8pt}

\begin{algtab*}
$j \leftarrow 0$

\algrepeat 
	$j \leftarrow j + 1$ \\
	Choose a block $\tau$ uniformly at random from $T$ \\
	$\vct{x}_j \leftarrow \vct{x}_{j-1} + (\mtx{A}_\tau)^\pinv (\vct{b}_\tau - \mtx{A}_\tau\vct{x}_{j-1})$ \hfill \{ Solve least-squares problem \} \\ 
\alguntil{$\enormsq{ \mtx{A}\vct{x}_j - \vct{b} } \leq \eps^2$}
$\hat{\vct{x}} \leftarrow \vct{x}_j$
\end{algtab*}
\end{minipage}}
\end{center}
\end{algorithm}
\end{center}

\subsection{Desiderata for the Partition} \label{sec:desiderata}

The implementation and behavior of the block Kaczmarz method
depend heavily on the properties of the submatrices $\mtx{A}_\tau$ indexed by the
blocks $\tau$ in the partition $T$.  Let us explain how the structure of the 
submatrices plays a role in the implementation; the claims about performance will
emerge from the theoretical results in Section~\ref{sec:intro-main}.

The most expensive (arithmetic) step in Algorithm~\ref{alg:randomized-block-kaczmarz}
occurs when we apply the pseudoinverse $\mtx{A}_{\tau}^\pinv$ to a vector.
We can perform this calculation efficiently \emph{provided that}
each submatrix $\mtx{A}_{\tau}$ has well-conditioned rows.
Indeed, in this case, we can invoke an iterative least-squares solver~\cite{Bjo96:Numerical-Methods},
such as {\tt CGLS},
to apply the pseudoinverse $\mtx{A}_\tau^\pinv$
approximately using a small number of matrix--vector
multiplies with $\mtx{A}_{\tau}$ and $\mtx{A}_{\tau}^*$.
In particular, we never need to form the pseudoinverse.



%
%
%

This observation highlights how important it is to control the geometric properties
of the submatrices $\mtx{A}_{\tau}$ induced by $T$.  Let us make a definition
that encapsulates the information 
that we will need.




\begin{definition}[Row Paving]
An $(m, \alpha, \beta)$ \emph{row paving} of a matrix $\mtx{A}$ is a partition $T = \{ \tau_1, \dots, \tau_m \}$ of the row indices that verifies
$$
\alpha \leq \lambda_{\min}( \mtx{A}_\tau \mtx{A}_{\tau}^* )
\quad\text{and}\quad
\lambda_{\max}( \mtx{A}_{\tau} \mtx{A}_{\tau}^* ) \leq \beta
\quad\text{for each $\tau \in T$.}
$$
The number $m$ of blocks is called the \emph{size} of the paving.  The numbers $\alpha$ and $\beta$ are called \emph{lower} and \emph{upper paving bounds}.  The ratio $\beta / \alpha$ gives a uniform bound on the squared condition number $\kappa^2(\mtx{A}_{\tau})$ for each $\tau$.  Note that $\alpha = 0$ unless each submatrix $\mtx{A}_{\tau}$ is fat.
\end{definition}

Every partition $T$ of the rows of a matrix $\mtx{A}$ has associated paving parameters $(m, \alpha, \beta)$.  In a moment, we will see how these quantities play a role in the performance of the algorithm.  Roughly speaking, it is best that the size $m$, the upper bound $\beta$, and the conditioning $\beta/\alpha$ of the paving are small.  Later, in Sections~\ref{sec:intro-paving} and~\ref{sec:pavings}, we will discuss what kind of bounds we can expect on the paving parameters, as well as computational methods for producing good pavings.  Note that, for a row paving to be useful in our context, the cost of producing the paving must not exceed the cost of solving the least-squares problem by other means!

%
%
%

\subsection{Convergence of Randomized Block Kaczmarz} \label{sec:intro-main}

The main result of this paper provides information about the convergence properties of the randomized block Kaczmarz method, Algorithm~\ref{alg:randomized-block-kaczmarz}, in terms of the parameters of the row paving $T$.

\begin{theorem}[Convergence] \label{thm:convergence}
Suppose $\mtx{A}$ is a matrix with full column rank that admits an $(m, \alpha, \beta)$ row paving $T$.  Consider the least-squares problem
$$
\mathrm{minimize} \quad \enormsq{ \mtx{A} \vct{x} - \vct{b} }.
$$
Let $\vct{x}_{\star}$ be the unique minimizer, and define the residual $\vct{e} := \mtx{A} \vct{x}_{\star} - \vct{b}$.
For any initial estimate $\vct{x}_0$, the randomized block Kaczmarz method, Algorithm~\ref{alg:randomized-block-kaczmarz}, produces a sequence $\{\vct{x}_j : j \geq 0 \}$ of iterates that satisfies
\begin{equation} \label{eqn:block-rate}
\Expect \enormsq{ \vct{x}_j - \vct{x}_{\star} }
	\ \leq \ \left[ 1 - \frac{\sigma_{\min}^2(\mtx{A})}{\beta m} \right]^{j}
	\enormsq{ \vct{x}_0 - \vct{x}_{\star} }
	\ + \ \frac{\beta}{\alpha} \cdot \frac{\enormsq{\vct{e}}}{ \sigma_{\min}^2(\mtx{A}) }.
\end{equation}
\end{theorem}

\noindent
Turn to Section~\ref{sec:results} for the proof of Theorem~\ref{thm:convergence}.

The expression~\eqref{eqn:block-rate} states that the block Kaczmarz method exhibits an expected linear rate of convergence until it reaches a ball about the true solution.  The radius of this ball, which we call the \emph{convergence horizon}, is comparable with the second term on the right-hand side of~\eqref{eqn:block-rate}.  The bracket controls the convergence rate.  The minimum singular value of $\mtx{A}$ affects both the rate of convergence and the convergence horizon.  In each case, we prefer $\sigma_{\min}(\mtx{A})$ to be as large as possible.

The properties of the row paving play an interesting role in Theorem~\ref{thm:convergence}.  Curiously, the rate of convergence depends only on the upper paving bound $\beta$ and the number $m$ of blocks in the paving.  On the other hand, the convergence horizon reflects the conditioning $\beta / \alpha$ of the paving.  Thus, the conditioning of the paving only affects the error bound when the least-squares problem is inconsistent (i.e., $\vct{e}$ is nonzero).
Nevertheless, as Section~\ref{sec:desiderata} suggests, we usually want the paving to be well conditioned to ensure that we can apply the block update rule~\eqref{eqn:block-kacz} efficiently.



%
%
%

\subsection{Simple Kaczmarz versus Block Kaczmarz}

First, notice that Theorem~\ref{thm:convergence} improves on the earlier result~\eqref{eqn:Needell-rate} for the simple Kaczmarz method.  Indeed, the simple Kaczmarz method is equivalent to using a row paving with $n$ blocks, where each block contains exactly one of the $n$ rows.  When $\mtx{A}$ is standardized, the paving constants satisfy $\alpha = \beta = 1$, and we reach the error bound
$$
\Expect \enormsq{ \vct{x}_j - \vct{x}_{\star} }
	\ \leq \ \left[ 1 - \frac{\sigma_{\min}^2(\mtx{A})}{n} \right]^{j}
	\enormsq{ \vct{x}_0 - \vct{x}_{\star} }
	\ + \ \frac{\enormsq{\vct{e}}}{ \sigma_{\min}^2(\mtx{A}) }.
$$
The convergence horizon $\normsq{\vct{e}} \leq n \infnormsq{\vct{e}}$, so the displayed bound beats~\eqref{eqn:Needell-rate} when $\mtx{A}$ is standardized.

More generally,
suppose $\mtx{A}$ is a standardized matrix with an $(m, \alpha, \beta)$
row paving $T$.  Let us compare the simple Kaczmarz algorithm with uniformly random control (Section~\ref{sec:simple-kacz}) to the block Kaczmarz method, Algorithm~\ref{alg:randomized-block-kaczmarz}.  Both methods satisfy an error bound of the form
$$
\Expect \enormsq{\vct{x}_j - \vct{x}_{\star}}
	\ \leq \ \cnst{e}^{-j \rho} \cdot \enormsq{\vct{x}_0 - \vct{x}_{\star}} \ + \ h
$$
where the convergence rate $\rho$ and the convergence horizon $h$ depend on the choice of algorithm.


First, we compare the convergence rates of the two methods.  The bounds~\eqref{eqn:Needell-rate}
and~\eqref{eqn:block-rate} imply
\begin{equation} \label{eqn:compare-rates}
\rho_{\rm simp} \geq \frac{\sigma_{\min}^2(\mtx{A})}{n}
\quad\text{and}\quad
\rho_{\rm block} \geq \frac{\sigma_{\min}^2(\mtx{A})}{\beta m}
\end{equation}
because $\log(1 - t) \leq -t$ when $t < 1$.
In other words, the simple method requires a factor
$n/(\beta m)$ more iterations than the block method to achieve the
same reduction in error.


Next, we examine the convergence horizons.  The bounds~\eqref{eqn:Needell-rate}
and~\eqref{eqn:block-rate} yield
$$
h_{\rm simp} = \frac{n \infnormsq{\vct{e}}}{\sigma_{\min}^2(\mtx{A})}
\quad\text{and}\quad
h_{\rm block} = \frac{\beta}{\alpha} \cdot \frac{\enormsq{\vct{e}}}{\sigma_{\min}^2(\mtx{A})}.
$$
Their ratio satisfies
$$
\frac{h_{\rm block}}{h_{\rm simp}}
	= \frac{\beta}{\alpha} \cdot \frac{\enormsq{\vct{e}}}{n \infnormsq{\vct{e}}}
	\leq \frac{\beta}{\alpha}.
$$
We see that the convergence horizon $h_{\rm block}$ for the block method never exceeds $h_{\rm simp}$
by a factor larger than the conditioning $\beta/\alpha$ of the row paving.  But $h_{\rm block}$ may
be substantially smaller than $h_{\rm simp}$ if the components of the residual vector are highly nonuniform.

To make more detailed claims about the relative merits of the two algorithms, we describe two situations where we have additional information about the structure of the matrix $\mtx{A}$, or a lack thereof.


\subsubsection{Example 1: Fast Updates} \label{sec:fast-update}

First, we consider the case where the computational cost of the block update
rule~\eqref{eqn:block-kacz} is roughly comparable with the cost of the
simple update rule~\eqref{eqn:simple-kacz}.  This situation can occur when
\begin{itemize}
\item	Each submatrix $\mtx{A}_{\tau}$ admits a fast multiply; and

\item	Each submatrix $\mtx{A}_{\tau}$ is well conditioned.
\end{itemize}
See Section~\ref{sec:experiments-circ} for a numerical example where these properties hold.

In this setting, one iteration of the block method has roughly the same
cost as one iteration of the standard method.  As a consequence, the
comparison~\eqref{eqn:compare-rates} of the convergence rates provides
a reasonable assessment of how much each algorithm reduces the error per unit of
arithmetic.  We see that the block algorithm really is about $n/(\beta m)$
times faster than the simple Kaczmarz method.  When $\beta m$ is small in
comparison with $n$, this represents a massive acceleration.

\subsubsection{Example 2: Unstructured Submatrices} \label{sec:unstruct}

On the other hand, suppose that each submatrix $\mtx{A}_{\tau}$ is unstructured and dense.
Then the block method may involve much more arithmetic per iteration than the
simple method.  As a consequence, the comparison~\eqref{eqn:compare-rates}
is unfair to the simple method.

In this case, it is more appropriate to examine the convergence rate
per \emph{epoch}, the minimum number of iterations it takes the algorithm
to touch each row of $\mtx{A}$ once.
For the simple Kaczmarz method, an epoch consists of $n$ iterations;
for the block method, an epoch consists of $m$ iterations.
In this setting, each algorithm requires about the same amount of arithmetic in one epoch,
so we consider the per-epoch convergence rates:
$$
m \cdot \rho_{\rm block} \geq \frac{\sigma^2_{\min}(\mtx{A})}{\beta}
\quad\text{and}\quad
n \cdot \rho_{\rm simp} \geq \sigma^2_{\min}(\mtx{A}).
$$
We see that, in theory, the per-epoch convergence of the block method
is worse, and the disadvantage increases with the upper bound $\beta$
on the paving.  The best case for the block method occurs if $\beta = 1$.
This may happen, for example, when each block in the paving contains
a single row of the matrix ($m=n$).

In practice, the block method displays better convergence behavior
than this estimate suggests.  The block method accrues
further advantages because of subtle computational issues involving
data transfer and basic linear algebra subroutines ({\tt BLASx}).
We discuss these points in Section~\ref{sec:why}, and we provide 
some numerical support in Section~\ref{sec:experiments-dense}.


%
%

\subsection{Existence of Good Pavings} \label{sec:intro-paving}

So far, we have assumed that the matrix $\mtx{A}$ comes packaged with a
natural row paving $T$.  For some of the applications we have in mind,
this hypothesis is reasonable.  Nevertheless, the block Kaczmarz method would
be more versatile if we could construct row pavings for a broad class
of matrices.  To that end, Popa~\cite{Pop99:Block-Projections-Algorithms}
has developed an approach for producing a paving of a sparse matrix;
see also~\cite{Pop01:Fast-Kaczmarz-Kovarik,Pop04:Kaczmarz-Kovarik-Algorithm}.
But we can travel much farther down this road.
It is an astonishing fact that \emph{every}
standardized matrix admits a good row paving.

\begin{proposition}[Existence of Good Row Pavings] \label{prop:intro-paving}
Fix a number $\delta \in (0, 1)$.  Let $\mtx{A}$ be a standardized matrix with $n$ rows.
Then $\mtx{A}$ admits a row paving whose parameters satisfy
$$
m \leq \cnst{C_{pave}} \cdot \delta^{-2} \normsq{\mtx{A}} \log(1+n)
\quad\text{and}\quad
1 - \delta \leq \alpha \leq \beta \leq 1 + \delta.
$$
The number $\cnst{C_{pave}}$ is a positive, universal\,%
\footnote{A \emph{universal} constant has no dependence on any parameter.}
constant.
\end{proposition}


Proposition~\ref{prop:intro-paving} follows most directly
from the recent results~\cite[Cor.~1.5]{Ver06:Random-Sets}
and~\cite[Thm.~1.2]{Tro09:Column-Subset},
whose provenance can be traced to the celebrated
papers~\cite{BT87:Invertibility-Large,BT91:Problem-Kadison-Singer}.
Although Proposition~\ref{prop:intro-paving} is only an existential result, the literature describes several efficient algorithms for constructing row pavings.  In particular, under some additional conditions, it is possible to pave a matrix by partitioning its rows \emph{at random}.  See Section~\ref{sec:pavings} for more results and background on paving.

\subsubsection{Understanding the Paving Theorem}

Before we continue, let us take a moment to explain some of the key aspects of Proposition~\ref{prop:intro-paving}.
The main point is that the size of the paving depends only on the spectral norm of the matrix---not on the smallest singular value.  As a consequence, it is possible to pave matrices with substantial null spaces!

A second point is that we can make the squared conditioning $\beta / \alpha$ as close to one as we desire.  In particular, the choice $\delta = 0.5$ yields $\beta/\alpha \leq 3$.  This property licenses us to apply an iterative algorithm to solve least-squares problems involving the submatrices induced by the paving, as described in Section~\ref{sec:desiderata}.  We remark that the dependence of the paving size $m$ on the parameter $\delta$ is optimal%
\footnote{To verify this point, consider a large matrix $\mtx{A}$ whose entries have equal magnitude and independent random signs.  Use the Bai--Yin Law~\cite[Thm.~2]{BY93:Limit-Smallest} to estimate singular values and norms.}
as $\delta \to 0$.


Next, we develop some intuition about the role of the spectral norm in Proposition~\ref{prop:intro-paving}.  Suppose that $\mtx{A}$ is a standardized matrix with $n$ rows.  If the lower paving bound $\alpha > 0$, then a row paving $T$ of $\mtx{A}$ must contain at least $n / \rank(\mtx{A})$ blocks.  Otherwise, $\mtx{A}_{\tau}$ is rank deficient for some $\tau \in T$ simply because this submatrix has more than $\rank(\mtx{A})$ rows.  Now, using the fact $\fnormsq{\mtx{A}} \leq \rank(\mtx{A}) \normsq{\mtx{A}}$, we easily verify that 
$$
\normsq{\mtx{A}} \geq \frac{n}{\rank(\mtx{A})}.
$$
This bound is sharp.  (Consider the two extreme examples: a matrix with orthonormal rows and a matrix with identical rows.)  Therefore, we can view the squared spectral norm as a proxy for the minimal number of blocks in a row paving whose lower bound $\alpha > 0$.

We conclude that Proposition~\ref{prop:intro-paving} delivers a row paving whose size falls within a logarithmic factor of optimal.  One may wonder whether it is possible to remove the logarithm.  For general matrices, this question remains open. It is known~\cite{And79:Extensions-Restrictions,BHKW88:Matrix-Norm,BT91:Problem-Kadison-Singer} that an affirmative answer would imply the long-standing conjecture of Kadison and Singer~\cite{KS59:Extensions-Pure}.

%


\subsection{Paved with Good Intentions}

We conclude the Introduction by merging our theorem on the convergence of the block Kaczmarz method with the result on the existence of pavings.

\begin{corollary}[Block Kaczmarz with a Good Row Paving] \label{cor:good-intentions}
Suppose that $\mtx{A}$ is a standardized matrix with full column rank.  Let $T$ be a good row paving of $\mtx{A}$, as guaranteed by Proposition~\ref{prop:intro-paving} with $\delta = 1/2$.  Under the notation of Theorem~\ref{thm:convergence}, the block Kaczmarz method, Algorithm~\ref{alg:randomized-block-kaczmarz} admits the convergence estimate
$$
\Expect \normsq{ \vct{x}_j - \vct{x}_{\star} }
	\leq \left[ 1 - \frac{1}{6 \cnst{C_{pave}} \kappa^2(\mtx{A}) \log(1+n)} \right]^j
	\enormsq{ \vct{x}_0 - \vct{x}_\star }
	\ + \
	\frac{3\enormsq{\vct{e}}}{\sigma_{\min}^2(\mtx{A})}
$$
where $\cnst{C_{pave}}$ is the constant from Proposition~\ref{prop:intro-paving}.
\end{corollary}

To summarize, Proposition~\ref{prop:intro-paving} yields a small paving of the matrix $\mtx{A}$ with exceptional conditioning.  With this choice of paving, we can perform the block Kaczmarz update~\eqref{eqn:block-kacz} quickly using an iterative least-squares algorithm.  (Indeed, to apply $\mtx{A}_\tau^\pinv$ with a fixed level of precision, it suffices to perform a constant number of matrix--vector multiplies with $\mtx{A}_{\tau}$ and $\mtx{A}_{\tau}^*$.)  Furthermore, we see that the block Kaczmarz method converges linearly with a rate that is controlled by the condition number of the matrix $\mtx{A}$, and the convergence horizon is on the same order as the size of the residual.  This is essentially the best outcome one might hope for.


\subsection{Organization}

The rest of the paper has the following structure.  Section~\ref{sec:results} contains a proof of the main result, Theorem~\ref{thm:convergence}.  In Section~\ref{sec:pavings}, we give an overview of the literature on pavings.  Section~\ref{sec:implementation} discusses numerical aspects of the block Kaczmarz method.  We discuss related work on Kaczmarz methods and future directions in Section~\ref{sec:future}.  Finally, Appendix~\ref{app:fit} offers a proof of a supplemental result.


\section{Analysis of the Randomized Block Kaczmarz Algorithm}
\label{sec:results} 

This section contains the proof of the main result, Theorem~\ref{thm:convergence},
on the convergence of Algorithm~\ref{alg:randomized-block-kaczmarz}.  We commence
with two simple lemmas.  The first step provides a deterministic bound on how
much one iteration of the algorithm reduces the error.


\begin{lemma} \label{lem:err-bd}
Instate the hypotheses and notation of Theorem~\ref{thm:convergence}.
Then the error at iteration $j$ satisfies the deterministic bound
$$
\enormsq{ \vct{x}_{j} - \vct{x}_{\star} }
	\leq \enormsq{ (\Id - \mtx{A}_{\tau}^\pinv \mtx{A}_{\tau})(\vct{x}_{j-1} - \vct{x}_{\star}) }
	\ + \
	\frac{1}{\alpha} \enormsq{ \vct{e}_{\tau} },
$$
where $\tau = \tau(j)$ is the block selected at iteration $j$.
\end{lemma}

\begin{proof}
According to the update rule~\eqref{eqn:block-kacz}, block Kaczmarz computes
$$
\vct{x}_{j}
	= \vct{x}_{j-1} + \mtx{A}_{\tau}^\pinv (\vct{b}_{\tau} - \mtx{A}_{\tau} \vct{x}_{j-1})
	= \vct{x}_{j-1} + \mtx{A}_{\tau}^\pinv \mtx{A}_{\tau} (\vct{x}_\star - \vct{x}_{j-1})
	- \mtx{A}_{\tau}^\pinv \vct{e}_{\tau},
$$
where we have introduced the decomposition $\vct{b} = \mtx{A} \vct{x}_{\star} - \vct{e}$, restricted to the coordinates listed in $\tau$.
Subtract $\vct{x}_{\star}$ from both sides to obtain
$$
\vct{x}_{j} - \vct{x}_{\star}
	= (\Id - \mtx{A}_{\tau}^\pinv \mtx{A}_{\tau} )(\vct{x}_{j-1} - \vct{x}_\star)
	- \mtx{A}_{\tau}^\pinv \vct{e}_{\tau}.
$$
The range of $\mtx{A}_{\tau}^\pinv$ and the range of $\Id - \mtx{A}_\tau^\pinv \mtx{A}_\tau$ are orthogonal, so we may invoke the Pythagorean Theorem to reach
$$
\enormsq{ \vct{x}_{j} - \vct{x}_{\star} }
	= \enormsq{ (\Id - \mtx{A}_{\tau}^\pinv \mtx{A}_{\tau} )(\vct{x}_{j-1} - \vct{x}_\star) }
	+ \enormsq{ \mtx{A}_{\tau}^\pinv \vct{e}_{\tau} }.
$$
The second term on the right-hand side satisfies
$$
\enormsq{ \mtx{A}_{\tau}^\pinv \vct{e}_{\tau} }
	\leq \sigma_{\max}^2(\mtx{A}_{\tau}^\pinv) \enormsq{\vct{e}_{\tau}}
	\leq \frac{1}{\sigma_{\min}^2(\mtx{A}_{\tau})} \enormsq{\vct{e}_{\tau}}
	\leq \frac{1}{\alpha} \enormsq{\vct{e}_{\tau}},
$$
where $\alpha$ is the lower bound on the row paving $T$.
Combine the last two displays to wrap up.
\end{proof}

The second lemma gives us a means to average the two quantities appearing
in Lemma~\ref{lem:err-bd} over a random choice of the block $\tau = \tau(j)$.

\begin{lemma} \label{lem:project-rdm-block}
Instate the hypotheses and notation of Theorem~\ref{thm:convergence}.
Suppose that $\tau$ is chosen uniformly at random from the row paving $T$.
For fixed vectors $\vct{u}$ and $\vct{v}$, 
it holds that
$$
\Expect \enormsq{ (\Id - \mtx{A}_{\tau}^\pinv \mtx{A}_{\tau}) \vct{u} }
	\leq \left[ 1 - \frac{\sigma^2_{\min}(\mtx{A})}{\beta m} \right] \enormsq{\vct{u}}
\quad\text{and}\quad
\Expect \enormsq{ \vct{v}_{\tau} } = \frac{1}{m} \enormsq{\vct{v}}.
$$
\end{lemma}

\begin{proof}
The second identity emerges from a very short calculation:
$$
\Expect \enormsq{\vct{v}_\tau}
	= \frac{1}{m} \sum\nolimits_{\omega \in T} \enormsq{\vct{v}_{\omega}}
	= \frac{1}{m} \enormsq{\vct{v}},
$$
which depends on the fact that the blocks of $T$ partition the components of $\vct{v}$.

Since $\mtx{A}_{\tau}^\pinv \mtx{A}_{\tau}$ is an orthogonal projector, we may apply the Pythagorean Theorem to obtain the relation
$$
\Expect \enormsq{ (\Id - \mtx{A}_{\tau}^\pinv \mtx{A}_{\tau}) \vct{u} }
	= \enormsq{\vct{u}} - \Expect \enormsq{ \mtx{A}_{\tau}^\pinv \mtx{A}_{\tau} \vct{u} }.
$$
We control the remaining expectation as follows.
\begin{align*}
\Expect \enormsq{ \mtx{A}_{\tau}^\pinv \mtx{A}_{\tau} \vct{u} }
	&\geq \Expect \left[ \sigma_{\min}^2(\mtx{A}_{\tau}^\pinv) \enormsq{ \mtx{A}_{\tau} \vct{u} } \right]
	\geq \frac{1}{\beta} \cdot \Expect \enormsq{ \mtx{A}_{\tau} \vct{u} } \\
	&= \frac{1}{\beta m} \sum_{\omega \in T} \enormsq{ \mtx{A}_{\omega} \vct{u} }
	= \frac{1}{\beta m} \enormsq{ \mtx{A} \vct{u} }
	\geq \frac{\sigma^2_{\min}(\mtx{A})}{\beta m} \enormsq{\vct{u}}.
\end{align*}
The second inequality depends on the bound $\sigma_{\min}^2(\mtx{A}_{\tau}^\pinv) = \sigma_{\max}^{-2}(\mtx{A}_{\tau}) \geq \beta^{-1}$.  The fourth relation holds because the blocks in a paving partition the row indices of $\mtx{A}$.  To complete the proof, we simply combine the last two displays.
\end{proof}

The main result follows quickly once we merge the two lemmas.

\begin{proof}[Proof of Theorem~\ref{thm:convergence}]
First, we bound the expected error at iteration $j$ in terms of the error at iteration $j - 1$.
Average the bound from Lemma~\ref{lem:err-bd} over the randomness in $\tau = \tau(j)$
to reach
\begin{align*}
\Expect_{\tau} \enormsq{ \vct{x}_{j} - \vct{x}_{\star} }
	&\leq \Expect_{\tau}
	\enormsq{ (\Id - \mtx{A}_{\tau}^\pinv \mtx{A}_{\tau} )(\vct{x}_{j-1} - \vct{x}_\star) }
	+ \frac{1}{\alpha} \cdot \Expect_{\tau} \enormsq{\vct{e}_{\tau}} \\
	&\leq \left[ 1 - \frac{\sigma^2_{\min}(\mtx{A})}{\beta m} \right]
	\enormsq{ \vct{x}_{j-1} - \vct{x}_{\star} }
	+ \frac{1}{\alpha m} \enormsq{\vct{e}}.
\end{align*}
The second inequality follows from Lemma~\ref{lem:project-rdm-block},
with $\vct{u} = \vct{x}_{j-1} - \vct{x}_{\star}$ and $\vct{v} = \vct{e}$.


By applying this result repeatedly, we can control the expected
error after $j$ iterations in terms of the initial error.
Abbreviating $\gamma := 1 - \sigma_{\min}^2(\mtx{A})/(\beta m)$,
we obtain the estimate
\begin{align*}
\Expect \enormsq{ \vct{x}_{j} - \vct{x}_{\star} }
	&= \Expect_{\tau(1)} \Expect_{\tau(2)} \cdots \Expect_{\tau(j)}
	\enormsq{ \vct{x}_{j} - \vct{x}_{\star} } \\
	&\leq \gamma^{j} \enormsq{ \vct{x}_0 - \vct{x}_{\star} }
	+ \frac{1}{\alpha m} \enormsq{\vct{e}}
	\left( \sum\nolimits_{i=0}^{j-1} \gamma^i \right) \\
	&\leq \gamma^{j} \enormsq{ \vct{x}_0 - \vct{x}_{\star} }
	+ \frac{1}{\alpha m(1 - \gamma)} \enormsq{\vct{e}}.
\end{align*}
Reintroduce the value of $\gamma$ in this expression, and simplify to complete the proof.
\end{proof}

%
%
%
%
%
%
%
%

\section{A Conversation about Pavings}\label{sec:pavings} 

As we have seen, the properties of the row paving $T$ of the matrix $\mtx{A}$ have a significant effect on the behavior of the block Kaczmarz method, Algorithm~\ref{alg:randomized-block-kaczmarz}.  It is natural to ask if every standardized matrix $\mtx{A}$ admits a good paving, and---if so---how we can exhibit such a paving.

This section summarizes the main results from the literature on row pavings, with particular attention to algorithmic techniques for constructing pavings.  We focus on two methods in particular.  The first approach, which is more general, extracts a well-conditioned row submatrix from $\mtx{A}$ and repeats this process until the paving is complete.  The second approach, which is more automatic, simply forms a random partition of the rows with an appropriate number of blocks.

For clarity, we only discuss row paving theorems for standardized matrices.  For general matrices, it is often more natural to consider an alternative definition of a paving where the rows of the matrix are reweighted.  For the reader's convenience, we include some citations that address the general case.

\begin{remark}[Paving a Square Matrix]
The operator theory literature uses the term \emph{paving} to refer to a partition $T = \{\tau_1, \dots, \tau_m \}$ of the coordinates of a \emph{square} matrix $\mtx{H}$ with a zero diagonal in which each diagonal block satisfies the bound
$$
\norm{ \mtx{H}_{\tau\tau} } \leq (1 - \delta) \norm{ \mtx{H} }
\quad\text{for each $\tau \in T$,}
$$
where the parameter $\delta \in (0, 1)$.  We can obtain row paving results for a standardized matrix $\mtx{A}$ by applying paving results for a square matrix to the hollow Gram matrix $\mtx{H} = \mtx{AA}^* - \Id$.  This approach generally leads to an estimate for the size $m$ of the paving that has an excessive dependence on the spectral norm of $\mtx{A}$.
\end{remark}

\subsection{Subset Selection Theorems}

The first approach to paving relies on a type of result called a \emph{subset selection theorem}.  This class of result asserts that, under appropriate conditions, a matrix contains a (large) set of rows with distinguished geometric properties.  Proposition~\ref{prop:intro-paving} depends on the following subset selection theorem.


\begin{proposition}[Subset Selection] \label{prop:subset-select}
Fix a number $\delta \in (0, 1)$.  Let $\mtx{A}$ be a standardized matrix with $n$ rows.  Then there exists a subset $\tau$ of row indices with the properties
$$
\abs{\tau} \geq \frac{\cnst{c_{ss}} \cdot \delta^2 n}{\normsq{\mtx{A}}}
\quad\text{and}\quad
1 - \delta \leq \lambda_{\min}(\mtx{A}_\tau \mtx{A}_{\tau}^*)
\leq \lambda_{\max}(\mtx{A}_\tau \mtx{A}_{\tau}^*) \leq 1 + \delta.
$$
The number $\cnst{c_{ss}}$ is a positive, universal constant.
\end{proposition}

Proposition~\ref{prop:subset-select} follows from~\cite[Thm.~1.2]{Tro09:Column-Subset}, once we track the parameter $\delta$ through the proof.  This result has been attributed to Bourgain and Tzafriri~\cite{BT91:Problem-Kadison-Singer}, but the earliest reference seems to be Vershynin's paper~\cite[Cor.~1.5]{Ver06:Random-Sets}.  See Section~\ref{sec:ss-related} for further background.

Proposition~\ref{prop:subset-select} ensures that each standardized matrix contains a large set of well-conditioned rows.  To construct a paving of a standardized matrix $\mtx{A}$ with row indices $\{1, \dots, n\}$, we apply this result to identify a large subset $\tau_1$ of the row indices.  We apply the same result to the set of remaining rows $\{1, \dots, n\} \setminus \tau_1$ to bite off another subset $\tau_2$, and so forth.  After
$$
m \leq \cnst{C_{pave}} \cdot \delta^{-2} \normsq{\mtx{A}} \log(1+n)
$$
steps, we have exhausted the entire matrix.  This argument yields Proposition~\ref{prop:intro-paving}.

The paper~\cite{Tro09:Column-Subset} contains an efficient computational method for identifying the subset $\tau$ promised by Proposition~\ref{prop:subset-select}.   This algorithm chooses a random set $\omega$ of rows from the matrix $\mtx{A}$ with twice the cardinality of the desired subset $\tau$.  Then it computes a matrix factorization of the submatrix $\mtx{A}_{\omega}$ that exposes a well-conditioned subset $\tau$ of rows inside $\omega$.

\subsubsection{Related Results} \label{sec:ss-related}

The literature contains a variety of subset selection theorems that can be used to construct pavings.  The first major result in this direction is the Restricted Invertibility Principle of Bourgain and Tzafriri~\cite[Thm.~1.2]{BT87:Invertibility-Large}, which guarantees that every standardized matrix contains a set of $\cnst{const} / \norm{\mtx{A}}^2$ rows whose minimal singular value is bounded away from zero.  In a similar vein, Kashin and Tzafriri~\cite{KT94:Some-Remarks} establish that every standardized matrix contains a set of $\cnst{const} / \norm{\mtx{A}}^{2}$ rows whose norm is constant.  Both results are based on random selection combined with matrix factorization.  Neither result yields a submatrix whose singular values lie arbitrarily close to one.

In another notable paper, Bourgain and Tzafriri~\cite[Cor.~1.2]{BT91:Problem-Kadison-Singer} prove that every standardized matrix contains a set of $\cnst{const} \cdot \delta^{2} /\norm{\mtx{A}}^{4}$ rows whose squared singular values fall in the range $[1 - \delta, 1 + \delta]$.  This theorem offers quantitative control on the conditioning of the submatrix, but it gives the wrong dependence on the spectral norm of the matrix.  Proposition~\ref{prop:subset-select} is based on the same type of argument as this result of Bourgain and Tzafriri.  The iterative argument we use to draw the paving result, Proposition~\ref{prop:intro-paving}, as a corollary of Proposition~\ref{prop:subset-select} also appears in the paper~\cite{BT91:Problem-Kadison-Singer}.

 

The last few years have witnessed some striking advances in this area.  Indeed, Spielman and Srivastava~\cite{SS12:Elementary-Proof} have recently invented an elementary proof of the Restricted Invertibility Principle.  Their method only involves linear algebra, and it leads to sharp constants.  Youssef~\cite[Thm.~4.2]{You12:Restricted-Invertibility} has adapted these ideas to obtain an elementary proof of the Kashin--Tzafriri theorem~\cite{KT94:Some-Remarks}.  These results are appealing because they construct the required subsets using an algorithmic procedure that admits a polynomial-time implementation.  See~\cite{Sri10:Spectral-Sparsification,Nao11:Sparse-Quadratic} for further exposition.

Vershynin~\cite{Ver01:Johns-Decompositions} has obtained a theory of subset selection for matrices that are not necessarily standardized.  In his results, the squared Frobenius norm $\fnormsq{\mtx{A}}$ plays the role of the number $n$ of rows.  Srivastava's dissertation~\cite[Chap.~3]{Sri10:Spectral-Sparsification} contains an algorithm for (weighted) subset selection that applies to general matrices.  See Youssef's works~\cite{You12:Restricted-Invertibility,You12:Note-Column} for the latest developments in this direction.

Finally, let us mention that subset selection theorems have applications throughout mathematics and engineering.  See the paper~\cite{CT06:Kadison-Singer-Problem} for a discussion and references.

\subsection{Randomized Methods for Paving}

The second approach to paving has the benefit of utmost simplicity, but it is more limited in scope.  The idea is to divide the rows of the matrix into random blocks of approximately equal size.  Under additional assumptions, each submatrix induced by this partition is likely to be well conditioned.  To describe this idea in more detail, we first introduce the concept of a random partition.

\begin{definition}[Random Partition]\label{def:rp}
Suppose that $\pi$ is a permutation on $\{1, 2, \ldots, n\}$, chosen uniformly at random.  For each $i = 1, 2, \ldots, m$, define the set
$$
\tau_i = \{\pi(k) : k = \lfloor (i-1)n/m\rfloor + 1, \lfloor (i-1)n/m\rfloor + 2, \ldots, \lfloor in/m\rfloor\}.
$$
It is clear that $T = \{\tau_1, \tau_2, \ldots, \tau_{m}\}$ is a partition of $\{1, 2, \ldots, n\}$ into $m$ blocks of approximately equal size.  We say that $T$ is a \textit{random partition} of $\{1, 2, \dots, n\}$ into $m$ blocks.
\end{definition}

For every standardized matrix, we can use a random partition to construct a paving whose upper bound $\beta$ is relatively small.

\begin{proposition}[Random Paving: Upper Bound] \label{prop:randpave-upper}
Let $\mtx{A}$ be a tall\,%
\footnote{A $p \times q$ matrix is \emph{tall} when $p \geq q$.},
standardized matrix with $n$ rows.  Consider a randomized partition $T$ of the row indices with $m \geq \normsq{\mtx{A}}$ blocks.  Then $T$ is a row paving with upper bound $\beta \leq 6\log(1 + n)$, with probability at least $1-n^{-1}$.
\end{proposition}

Proposition~\ref{prop:randpave-upper} results from an argument based on the matrix Chernoff inequality~\cite[Thm.~1.1]{Tro12:User-Friendly} and a union bound.  A model for this type of proof appears in the paper~\cite{Tro11:Improved-Analysis}.  We omit the details.

In contrast, if we wish to construct a paving with a nontrivial lower bound $\alpha$, we must place additional assumptions on the matrix.  An example~\cite[Ex.~2.2]{BT91:Problem-Kadison-Singer} of Bourgain and Tzafriri implies that we must assume the rows of the matrix $\mtx{A}$ are weakly correlated to obtain a random paving with $\alpha > 0$.  It is natural to carve out a class of matrices that meet this requirement.

\begin{definition}[Incoherence] \label{def:incoherence}
Suppose $\mtx{A}$ is a matrix with rows $\vct{a}_1, \vct{a}_2, \ldots, \vct{a}_n$.  We say that $\mtx{A}$ is \emph{incoherent} when
$$
\max_{i \neq \ell} \ \abs{\ip{ \vct{a}_i }{ \vct{a}_\ell }} \leq\frac{\cnst{c_{inc}}}{\log(1+n)}.
$$
The number $\cnst{c_{inc}}$ is a positive, universal constant.
\end{definition}

Incoherent matrices arise, for example, in signal processing problems~\cite{DH00:Uncertainty-Principles}.  Every incoherent, standardized matrix admits a random paving with controlled lower and upper bounds.

\begin{proposition}[Random Paving]\label{prop:randpave}
Suppose that $\mtx{A}$ is an incoherent, standardized matrix with $n$ rows.  Let $T$ be a random partition of the row indices into $m$ blocks where $m \geq \cnst{C_{rand}} \cdot \delta^{-2} \normsq{\mtx{A}} \log(1+n)$.  Then $T$ is a row paving of $\mtx{A}$ whose paving bounds satisfy $1-\delta \leq \alpha \leq \beta \leq 1 + \delta$, with probability at least $1-n^{-1}$.
\end{proposition}

Proposition~\ref{prop:randpave} follows from~\cite[Cor.~5.2]{Tro08:Norms-Random}, along with some standard arguments~\cite{BT91:Problem-Kadison-Singer,Tro08:Random-Paving}.  The paper~\cite{CD12:Invertibility-Submatrices} contains superior estimates for the constants in this analysis.  See Section~\ref{sec:randpave-related} for some further background.

Random paving is a striking idea because it is almost completely automatic.  Given a guarantee that the matrix $\mtx{A}$ is incoherent and an estimate for the spectral norm, we can obtain a good paving of the matrix without any further computation.


\subsubsection{The Fast Incoherence Transform}

Prima facie, random paving has limited applicability because the incoherence hypothesis in Proposition~\ref{prop:randpave} is rather stringent.  Fortunately, there is a fast linear transformation that can be used to convert any standardized matrix into an incoherent matrix that is nearly standardized.

\begin{definition}[Fast Incoherence Transform]
The \emph{fast incoherence transform} is the $n \times n$ random matrix $\mtx{S} = \mathbf{F}\mtx{E}$ where $\mathbf{F}$ is the $n \times n$ unitary discrete Fourier transform (DFT) and $\mtx{E}$ is an $n \times n$ diagonal matrix whose entries are independent Rademacher%
\footnote{A \emph{Rademacher} random variable takes the values $\pm 1$ with equal probability.}
random variables.
\end{definition}

The matrix $\mtx{S}$ is unitary, so it preserves Euclidean structure in a problem.  Furthermore, $\mtx{S}$ can be applied to a vector in $\bigO(n \log n)$ arithmetic operations by means of the FFT algorithm.  Thus, for an $n \times d$ matrix $\mtx{A}$, we can form the product $\mtx{SA}$ with $\bigO(nd \log n)$ arithmetic operations.  When the matrix $\mtx{A}$ is sparse or admits a fast matrix--vector multiply, it may be more appropriate to work directly with the factorized representation $\mtx{SA}$, because it inherits a fast multiply from its two factors.  To emphasize, in many contexts, it is unnecessary to form $\mtx{SA}$ explicitly.

The critical fact is that the fast incoherence transform converts a large class of standardized matrices into incoherent matrices that are nearly standardized.

\begin{proposition}[Fast Incoherence Transform] \label{prop:fit}
Suppose that $\mtx{A}$ is a standardized matrix with $n$ rows whose norm satisfies
\begin{equation} \label{eqn:fit-hyp}
\normsq{\mtx{A}} \leq \frac{\cnst{c_{fit}} \cdot n}{\log^3(1+n)}.
\end{equation}
Let $\mtx{S}$ be an $n \times n$ fast incoherence transform.  Then the matrix $\mtx{W} = \mtx{SA}$ satisfies the probability bound
$$
\Prob{ \max_{i, \ell}\abs{ \ip{\vct{w}_i}{ \vct{w}_\ell} - \delta_{i\ell}} \geq \frac{\cnst{c_{inc}}}{\log(1+ n)} }
	\leq \frac{1}{n},
$$
where $\vct{w}_i$ is the $i$th row of $\mtx{W}$ and $\delta_{i\ell}$ is the Kronecker delta.
\end{proposition}

We do not have a reference for Proposition~\ref{prop:fit}, but the result is probably not new.  Appendix~\ref{app:fit} offers a short proof based on the Hanson--Wright inequality~\cite{HW71:Bound-Tail} for Rademacher chaos.

Let us pause to examine the hypothesis~\eqref{eqn:fit-hyp}.  Recall that, for a standardized matrix $\mtx{A}$ with $n$ rows, the squared spectral norm $\normsq{\mtx{A}}$ attains its maximal value $n$ when the rows are identical.  Therefore, the bound~\eqref{eqn:fit-hyp} stipulates that the rows of $\mtx{A}$ must exhibit a small amount of diversity.  In this case, the fast incoherence transform spreads out the diversity evenly so that no pair of columns is correlated too strongly.

\subsubsection{Incoherence, Paving, Kaczmarz} \label{sec:inc-pave-kacz}

This discussion suggests the following approach for solving the overdetermined least-squares problem~\eqref{eqn:basic-ls} with a standardized matrix $\mtx{A}$:
\begin{enumerate}
\item	Apply the fast incoherence transform $\mtx{S}$ to the objective function:
$$
\enormsq{ \mtx{A} \vct{x} - \vct{b} }
	= \enormsq{ \mtx{SA} \vct{x} - \mtx{S}\vct{b} }
	=: \enormsq{ \tilde{\mtx{A}} \vct{x} - \tilde{\vct{b}} }.
$$
\item	Draw a random partition $T$ of the rows of $\tilde{\mtx{A}}$ with the number of blocks determined according to Proposition~\ref{prop:randpave}.
\item	Apply the randomized block Kaczmarz method, Algorithm~\ref{alg:randomized-block-kaczmarz}, with the matrix $\tilde{\mtx{A}}$, the right-hand side~$\tilde{\vct{b}}$, and the paving $T$ to solve the transformed problem.
\end{enumerate}
Provided that the hypothesis~\eqref{eqn:fit-hyp} is in force, Proposition~\ref{prop:fit} shows that the fast incoherence transform is likely to convert the matrix $\mtx{A}$ into an incoherent matrix $\tilde{\mtx{A}}$.  Proposition~\ref{prop:randpave} shows that we can obtain a good paving of $\tilde{\mtx{A}}$ from a random partition of the rows.  If these randomized procedures are successful, when we apply the block Kaczmarz algorithm to solve the transformed least-squares problem, we obtain the same convergence guarantees outlined in Corollary~\ref{cor:good-intentions}.

For a well-conditioned $n \times d$ matrix $\mtx{A}$ with $d \gg \log n$, we can use this approach to solve the least-squares problem to a fixed level of precision after $\bigO(nd \log n)$ arithmetic operations.  In theory, this bound is almost comparable with the conjugate gradient method, which requires $\bigO(nd)$ operations to achieve fixed precision in this setting.

\subsubsection{Related Results} \label{sec:randpave-related}

The idea that randomness might help us to construct a paving is already inherent in the Bourgain--Tzafriri paper~\cite{BT87:Invertibility-Large} on the Restricted Invertibility Principle, where they use randomized row selection and matrix factorization to perform subset selection.  A result~\cite[Thm.~2.3]{BT91:Problem-Kadison-Singer} in their subsequent paper demonstrates that, in the presence of incoherence assumptions, a random partition induces a row paving with $\cnst{const} \cdot \norm{\mtx{A}}^4 \log(1+n)$ blocks; the extra factorization step is not necessary.  Furthermore, under an incoherence assumption slightly stricter than Definition~\ref{def:incoherence}, they prove that a random partition induces a paving with $\cnst{const} \cdot \norm{\mtx{A}}^4$ blocks; no logarithmic factor is necessary.  See the paper~\cite{Tro08:Random-Paving} for a modern proof of the latter result.

The aforementioned theorems all yield the wrong dependence on the spectral norm in the size of a random row paving.  The pr{\'e}cis~\cite{Tro08:Norms-Random} shows how to obtain the correct quadratic dependence that is quoted in Proposition~\ref{prop:randpave}.  If we were to strengthen the incoherence requirement in Proposition~\ref{prop:randpave}, it is likely that we could remove the logarithmic factor from the size of the paving by adapting an argument~\cite[Prop.~2.7]{BT91:Problem-Kadison-Singer} of Bourgain and Tzafriri.  On the other hand, the logarithmic factor is necessary at the incoherence level we have imposed~\cite[Ex.~2.2]{BT91:Problem-Kadison-Singer}.

The fast incoherence transform is based on ideas of Ailon and Chazelle~\cite{AC09:Fast-Johnson-Lindenstrauss}, who use the random matrix $\mtx{S}$ to perform dimension reduction.  We believe that the first application of $\mtx{S}$ for randomized linear algebra appears in the paper Woolfe et al.~\cite{WLRT08:Fast-Randomized}, where they use this transform to aid in computing matrix decompositions.  See the works~\cite{AMT11:Blendenpik-Supercharging,HMT11:Finding-Structure,BG12:Improved-Matrix} for further results in this direction.  Liberty's dissertation~\cite{Lib09:Accelerated-Dense} describes other randomized maps that can play a similar role.

\section{Numerical Aspects of Block Kaczmarz} \label{sec:implementation}

The main goal of this paper is to study the theoretical properties of the block Kaczmarz method, but we believe that a short discussion about numerics can also provide some useful insights.  In Section~\ref{sec:why}, we outline some of the situations where we expect the block Kaczmarz method to outperform the simple Kaczmarz algorithm.  Afterward, in Section~\ref{sec:how}, we consider some of the questions that arise when implementing the block Kaczmarz method.  Finally, in Section~\ref{sec:experiments}, we offer some simple computational examples to illustrate our main points.

\subsection{Why Use Block Kaczmarz?} \label{sec:why}

It is natural to ask when it might be better to use the block Kaczmarz scheme instead of the simple Kaczmarz scheme.  The answer to this question depends heavily on the specific application at hand, as well as the architecture of the computers on which the algorithm is implemented.

First, least-squares problems sometimes involve a matrix that admits a natural row paving.  In this case, it may be advantageous to exploit the paving algorithmically.  For example, certain signal processing applications involve multi-sampling schemes, where we collect several batches of uniform time samples of a signal.  Each sample set produces a set of equations that is easy to solve.  The block Kaczmarz method provides an effective way to use this structure~\cite{FS95:Kaczmarz-Based-Approach}.  See Section~\ref{sec:experiments-circ} for a related numerical example.


Second, the block Kaczmarz algorithm can be implemented more efficiently than the simple Kaczmarz algorithm in many computer architectures.  This claim rests on two facts.  First, data transfer now plays a major role in the cost of numerical algorithms.  The block Kaczmarz algorithm is efficient in this regard, because it moves a large block of equations into working memory and operates with it for some time.  In contrast, the simple Kaczmarz method repeatedly transfers new equations into working memory.
Second, the block Kaczmarz algorithm can exploit high-level basic linear algebra subroutines ({\tt BLASx}).  Indeed, inner products dominate the arithmetic in the simple Kaczmarz method, while it is possible to implement the block Kaczmarz algorithm using matrix--vector products.  As a result, block Kaczmarz relies on {\tt BLAS2}, rather than {\tt BLAS1}.  A more detailed discussion of these issues falls outside the scope of this paper.

\subsection{Implementing Block Kaczmarz Methods} \label{sec:how}

The randomized block Kaczmarz method, Algorithm~\ref{alg:randomized-block-kaczmarz}, is easy to describe, but there remain several implementation issues that require attention.

\subsubsection{The Block Update Rule}

Most of the arithmetic in the algorithm occurs when we apply the pseudoinverse $\mtx{A}_\tau^\pinv$ to a vector in the update~\eqref{eqn:block-kacz}.  Equivalently, we must solve an (underdetermined) least-squares problem at each iteration.  The appropriate numerical method depends on a wide variety of issues~\cite{Bjo96:Numerical-Methods}, so we cannot give a universal prescription.
Here are some factors worth considering.
\begin{itemize}
\item	When the blocks in the paving are very small, a direct method based on {QR} decomposition or the {SVD} is likely to be fastest.  A direct method may also be appropriate when the conditioning $\beta/\alpha$ of the paving is high and the matrix is dense.

\item	In case the conditioning $\beta/\alpha$ of the paving is small, we recommend using an iterative method, such as {\tt CGLS}, {\tt LSQR}, or the Chebyshev semi-iterative method.  These techniques may also be appropriate for very sparse problems.  It is not necessary to run these iterations until they have converged fully, and the algorithms will benefit from warm starts provided by the convergence of the outer iteration.
\end{itemize}


\subsubsection{Setting the Convergence Tolerance}

Theorem~\ref{thm:convergence} implies that Algorithm~\ref{alg:randomized-block-kaczmarz} can reduce the error $\enormsq{\hat{\vct{x}} - \vct{x}_{\star}}$ to a level comparable with the convergence horizon.  Let us examine how this fact affects our choice of the convergence tolerance.

Combine the convergence criterion $\enormsq{ \mtx{A}\hat{\vct{x}} - \vct{b} } \leq \eps^2$ with the decomposition $\vct{b} = \mtx{A} \vct{x}_{\star} - \vct{e}$ to obtain
$$
\enormsq{ \mtx{A}(\hat{\vct{x}} - \vct{x}_{\star}) + \vct{e} } \leq \eps^2.
$$
The residual $\vct{e}$ is orthogonal to the range of $\mtx{A}$, so we can apply the Pythagorean theorem and bound the $\ell_2$ norm below to reach
\begin{equation*} 
\sigma_{\min}^2(\mtx{A}) \enormsq{ \hat{\vct{x}} - \vct{x}_{\star} }
	+ \enormsq{\vct{e}}
	\leq \eps^2.
\end{equation*}
It follows that we must set the convergence tolerance $\eps$ so that
\begin{equation} \label{eqn:conv-tol}
\eps^2 > \left[ 1 + \frac{\beta}{\alpha} \right] \enormsq{\vct{e}}.
\end{equation}
Otherwise, we have no guarantee that the algorithm will terminate.

\subsubsection{Checking for Convergence}

The convergence criterion $\normsq{\mtx{A}\hat{\vct{x}} - \vct{b}} \leq \eps^2$ may be somewhat expensive to verify because it involves a multiply with the full matrix $\mtx{A}$. As a consequence, it is better to check for convergence only on occasion.  For instance, if the paving consists of $m$ blocks, we might only evaluate the convergence criterion every $m$ iterations.

\subsubsection{Randomized Cyclic Control} \label{sec:rand-cyclic}

Finally, in practice, the block Kaczmarz algorithm is more effective if we use a randomized control scheme different from the one we have analyzed.  Recall that Algorithm~\ref{alg:randomized-block-kaczmarz} samples a uniformly random block at each iteration, independent of all previous choices.  Instead, we recommend using an alternative scheme that samples blocks \emph{without} replacement.  We can express this method formally as follows.

\begin{enumerate}
\item	For each $q = 0, 1, 2, \dots$, draw an independent, uniformly random permutation $\pi_q$ on the indices $\{1, \dots, m \}$ of the blocks.

\item	For each iteration $j \geq 1$, decompose $j = qm + r$, where $q$ and $r$ are nonnegative integers with $0 \leq r < m$.  Select the block $\tau(j) = \tau_{\pi_q(r + 1)}$.
\end{enumerate}

\noindent
In other words, at each epoch, we cycle through all the blocks in random order. 

In Section~\ref{sec:experiments-circ}, we offer some numerical evidence that this alternative control scheme is more effective than the approach used in Algorithm~\ref{alg:randomized-block-kaczmarz}.  At present, compelling explanations for this phenomenon are lacking.  See~\cite{RR12:Beneath-Valley} for some discussion and conjectures.



\subsection{Numerical Experiments}\label{sec:experiments} 

In this section, we present some numerical experiments to complement our discussions about the implementation and theoretical performance of the randomized block Kaczmarz method,
Algorithm~\ref{alg:randomized-block-kaczmarz}.  In Section~\ref{sec:experiments-circ},
we consider an example where the block method has a clear advantage over the simple method.
In Section~\ref{sec:experiments-dense}, we examine some other cases where the benefits
of the block method are more subtle.

\subsubsection{Submatrices with Fast Multiplies} \label{sec:experiments-circ}

First, we study the situation where the matrix $\mtx{A}$ comes with a natural partition
of the rows into well-conditioned blocks, each admitting a fast matrix--vector
multiply.
As we have discussed in Section~\ref{sec:fast-update},
the block method has a significant advantage over the simple method in this
setting.  Our experiments bear out this point.

We build a $300 \times 100$ matrix $\mtx{A}_{\rm circ}$ by stacking $15$ partial random circulant
matrices:
$$
\mtx{A}_{\rm circ} = \begin{bmatrix} \mtx{C}_1 \\ \vdots \\ \mtx{C}_{15} \end{bmatrix}
\quad\text{where}\quad
\mtx{C}_{i} = \mtx{R} \mathbf{F}^* \mtx{E}_i \mathbf{F}
\quad\text{for $i = 1, \dots, 15$.}
$$
In this expression,
\begin{itemize}
\item	$\mtx{R}$ is the restriction to the first $20$ coordinates,
\item	$\mathbf{F}$ is the $100 \times 100$ unitary DFT, and
\item	$\mtx{E}_{i}$ is a random diagonal matrix whose entries are independent Rademacher
random variables.  Each $\mtx{E}_i$ is drawn independently from the others.
\end{itemize}
This construction ensures that each $\mtx{C}_i$ is a section of a circulant matrix
with orthonormal rows.
As a consequence, the pseudoinverse equals the adjoint: $\mtx{C}_i^\pinv = \mtx{C}_i^*$.
Furthermore, we can apply either $\mtx{C}_i$ or $\mtx{C}_i^*$ to a vector quickly%
\footnote{Recall that an FFT or inverse FFT of length $d$ requires roughly $d \log_2(d)$ complex floating-point
operations (flops), although the precise count depends on arithmetic properties of the number $d$.}
using one FFT and one inverse FFT, each of length $d$.


Using this matrix, we perform a small {\sc Matlab} experiment to compare the behavior of randomized block Kaczmarz and randomized simple Kaczmarz.
\begin{enumerate}
\item	Draw a random matrix $\mtx{A}_{\rm circ}$ according to the recipe above, and fix it.  Let $T$
be the row partition with 15 blocks induced by the circulant structure. 
\item	Let $\vct{x}_{\star} = [1, \dots, 1]^*$, and form $\vct{b} = \mtx{A} \vct{x}_{\star}$.
Set the initial iterate $\vct{x}_0 = \vct{0}$.
\item	For each of 100 trials,
\begin{enumerate}
\item	Apply the simple Kaczmarz method from Section~\ref{sec:simple-kacz} to produce iterates $\{ \vct{x}_j^{\rm simp} : j \geq 0 \}$.
\item	Apply Algorithm~\ref{alg:randomized-block-kaczmarz} to produce iterates $\{ \vct{x}_j^{\rm block} : j \geq 0 \}$.
\end{enumerate}
\item	For each algorithm, at each iteration $j$, compute the minimum, median, and maximum value of the approximation error $\enorm{\vct{x}^{\rm alg}_j - \vct{x}_{\star}}$ over the 100 trials.
\end{enumerate}
In this experiment, we form the matrix $\mtx{A}_{\rm circ}$ in the first step and store it.  To perform the update~\eqref{eqn:simple-kacz}, the simple Kaczmarz method loads the required row from memory without doing any extra computation, so the cost of the simple update rule is just $4d$ complex flops, where $d = 100$.  The block Kaczmarz method uses the structure of the circulant blocks, as described in the previous paragraph, to perform the update~\eqref{eqn:block-kacz} with about $4 d \log_2(d) + 4d$ complex flops.

In Figure~\ref{fig:circ} [left panel], we plot the approximation error as a function of the number of flops expended by each algorithm.  The heavy lines indicate the median error over 100 trials, while the minimum and maximum errors describe the boundaries of the shaded region.  The block algorithm reduces the error to $10^{-11}$ in about $1.6 \times 10^6$ complex flops, while the simple algorithm requires about $3.2 \times 10^7$ complex flops to achieve the same result.  This amounts to a 20-fold reduction in the amount of arithmetic!

In Section~\ref{sec:rand-cyclic}, we claim that the block Kaczmarz algorithm is more effective when we use an alternative control scheme that samples blocks without replacement.  To illustrate this point, let us perform an experiment with the block circulant matrix using the same methodology as above.  Figure~\ref{fig:circ} [right panel] compares the performance of the block Kaczmarz method when we sample blocks with and without replacement.  This chart shows that sampling without replacement reduces the amount of computation by about 15\%.  In other experiments, we have seen even more substantial improvements.

\begin{figure}[t]
\begin{center}
\includegraphics[height=2.5in]{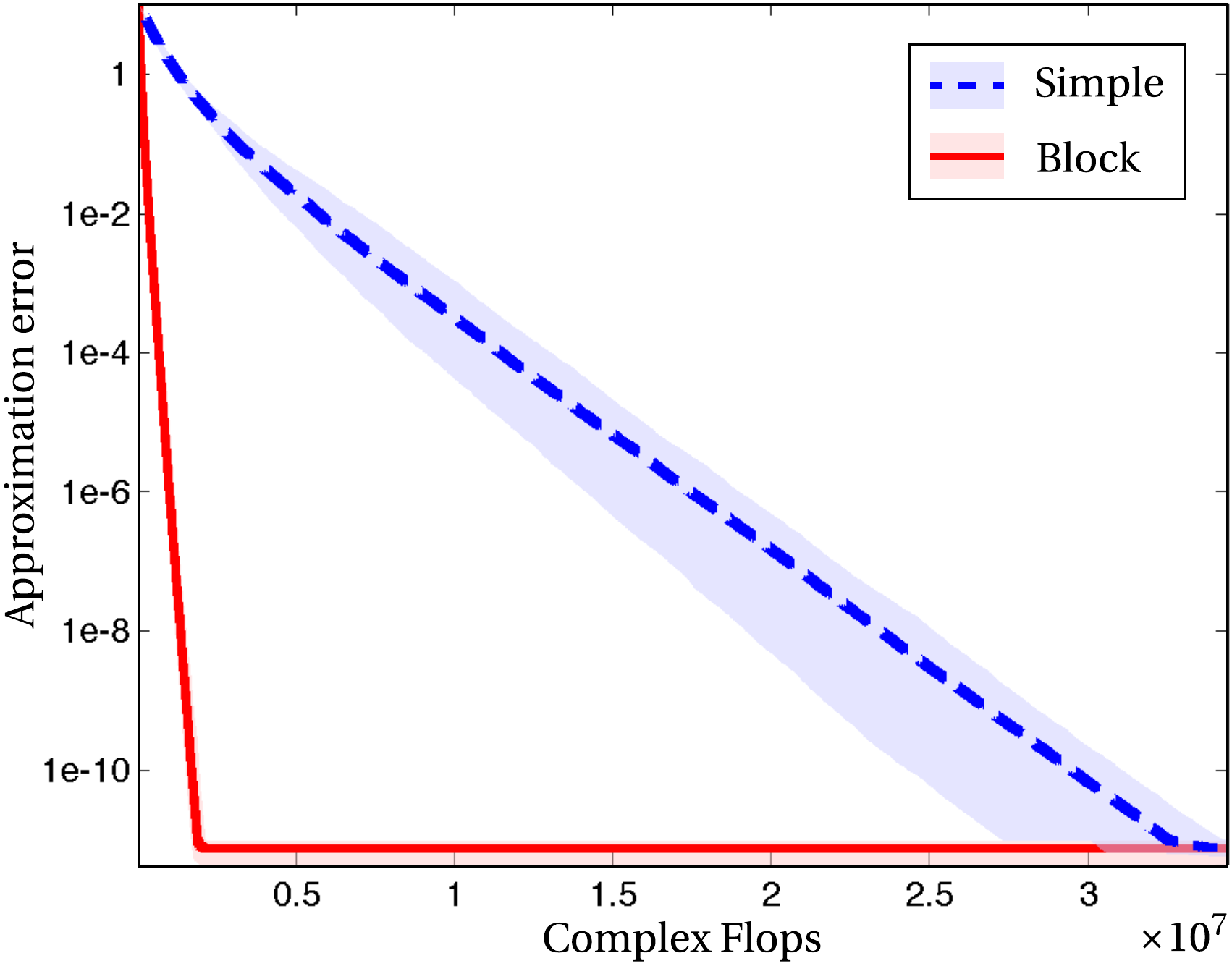}
\includegraphics[height=2.5in]{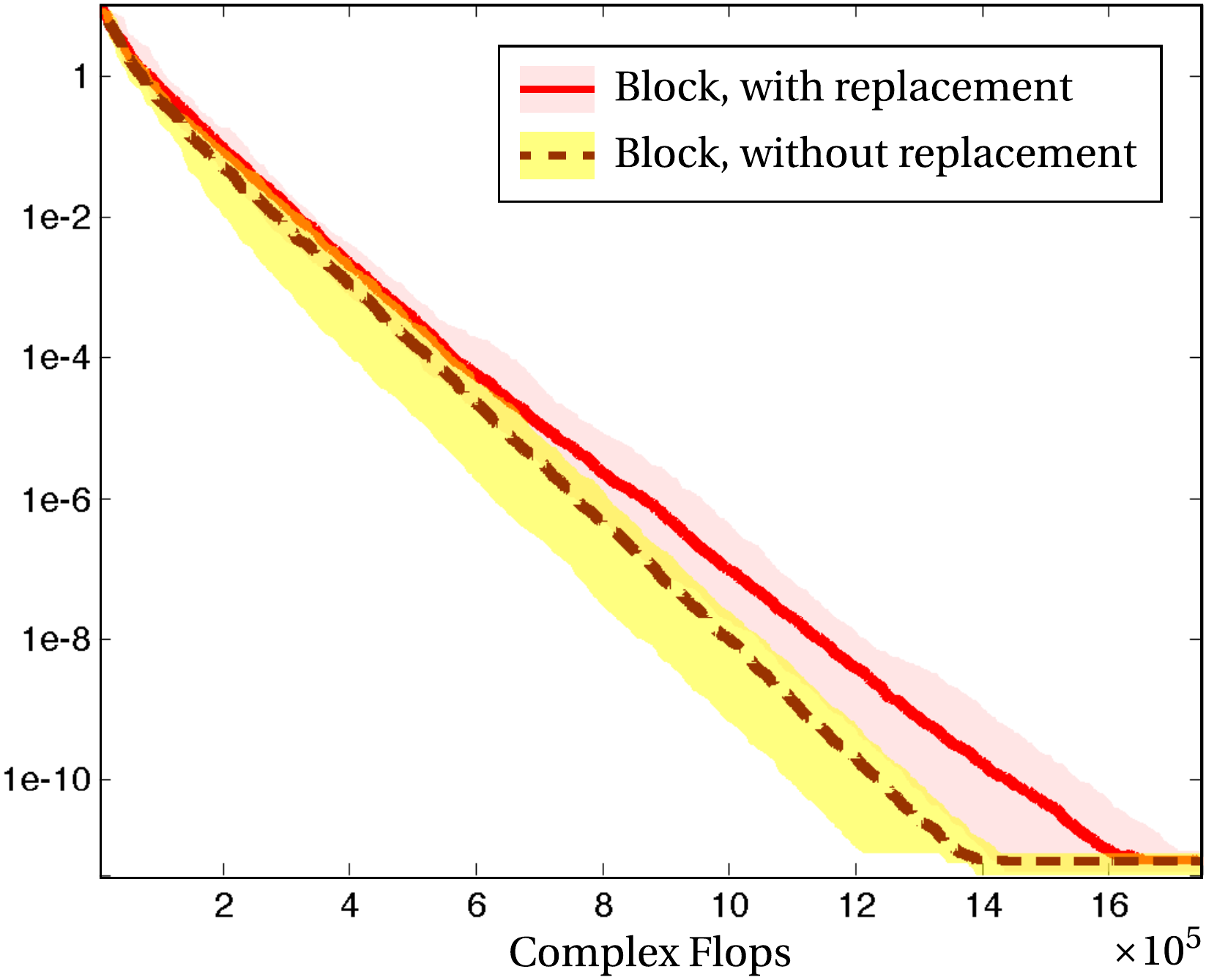}
\caption{({Convergence with a Block Circulant Matrix}){\bf .}
The matrix $\mtx{A}_{\rm circ}$ is a fixed $300 \times 100$ matrix consisting
of $15$ partial circulant blocks.
For each algorithm, we chart the approximation
error $\enorm{ \vct{x}_j - \vct{x}_{\star} }$ as a function of the
number of complex flops performed.
The heavy line denotes the median error over 100 independent trials;
the minimum and maximum errors envelop the shaded region.
See Section~\ref{sec:experiments-circ} for details.
{\bf [{Left}]} Convergence of the randomized block Kaczmarz method,
Algorithm~\ref{alg:randomized-block-kaczmarz}, versus the
randomized simple Kaczmarz method (Section~\ref{sec:simple-kacz}).
{\bf [{Right}]}  Convergence of two variants  (Section~\ref{sec:rand-cyclic})
of the randomized block Kaczmarz method.
One samples blocks independently with replacement;
the other samples blocks without replacement in each epoch.
} \label{fig:circ}
\end{center}
\end{figure}

\begin{figure}[t]
\begin{center}
\includegraphics[height=1.75in]{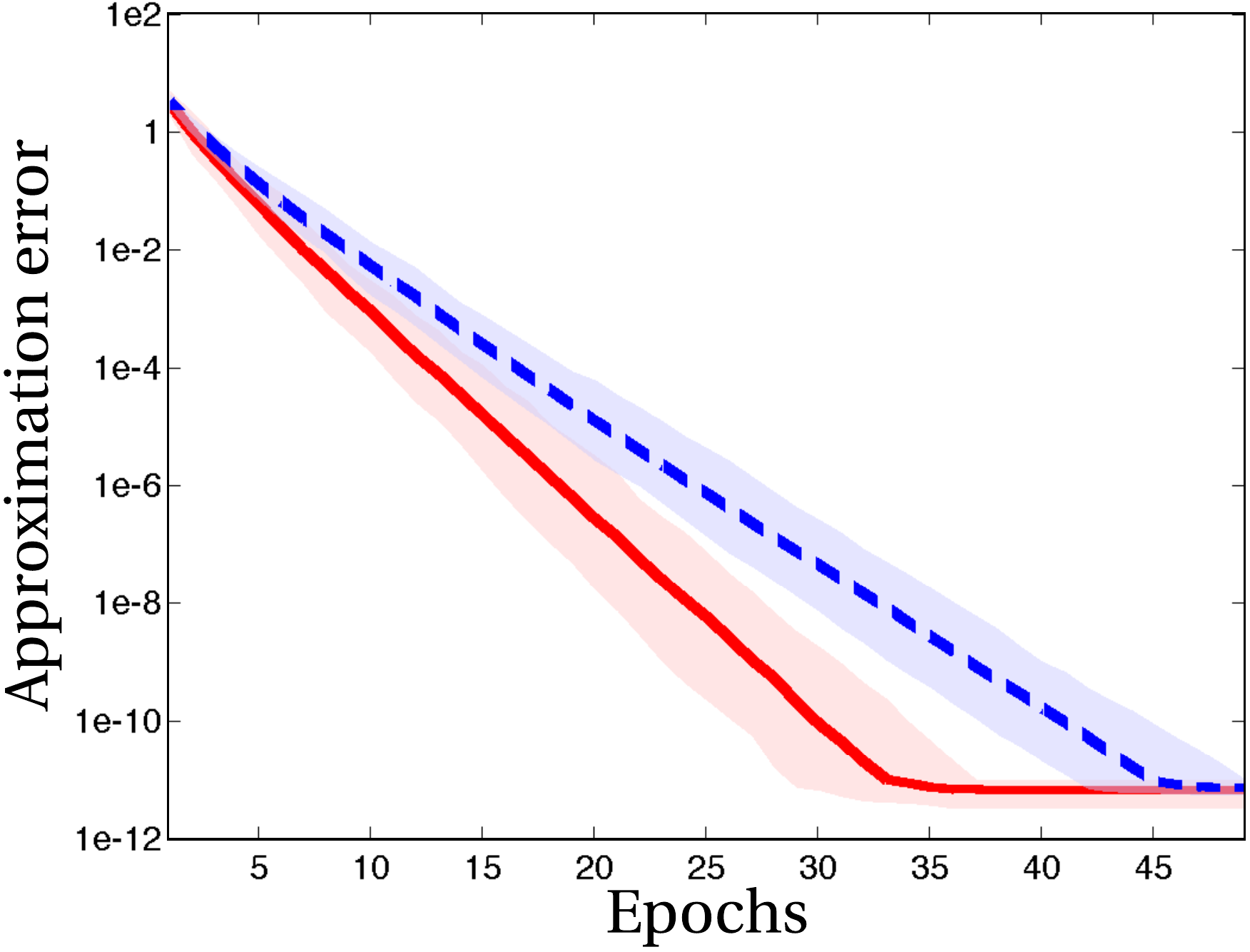}
\raisebox{0.025in}{\includegraphics[height=1.7in]{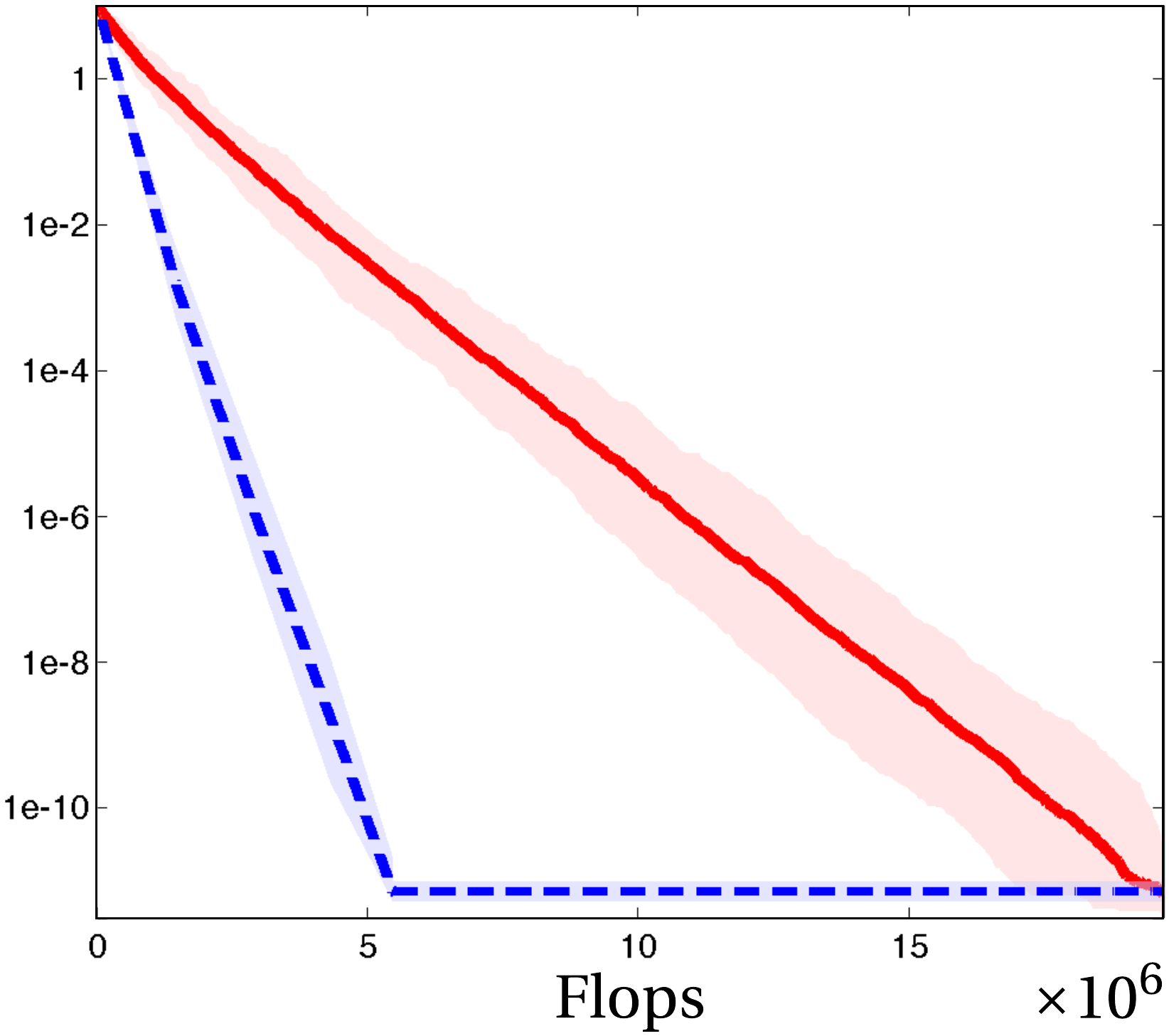}
\includegraphics[height=1.7in]{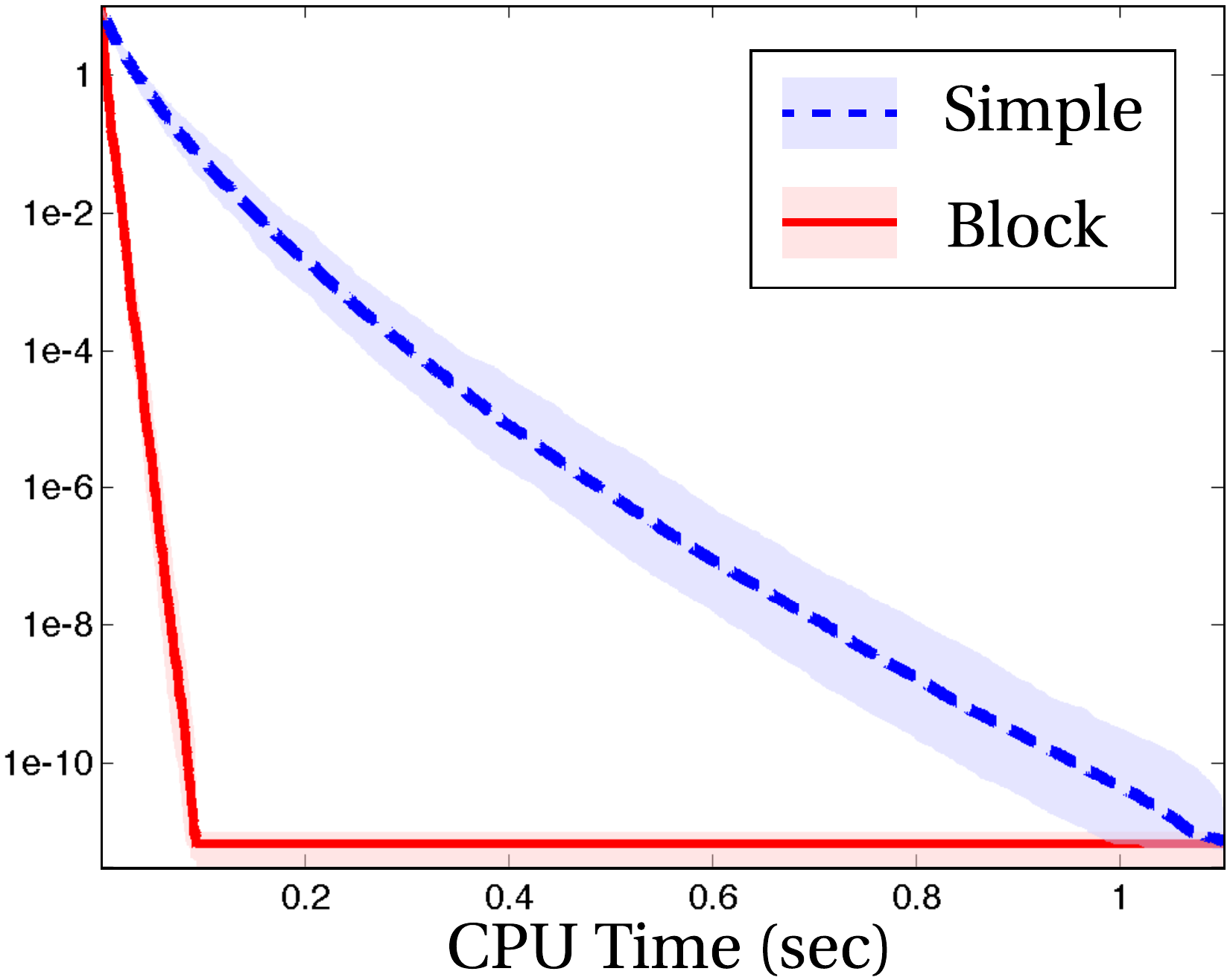}}
\caption{({Convergence with a Random Matrix}){\bf .}
The matrix $\mtx{A}_{\rm rdm}$ is a fixed $300 \times 100$ matrix whose rows are drawn
uniformly at random from the Euclidean unit sphere and partitioned into 10
blocks of equal size.
For the iterates generated by each algorithm, we plot the decay of the approximation error
$\enorm{\vct{x}_j - \vct{x}_{\star}}$
as a function of three computational resources.
The heavy line denotes the median error over 100 independent trials;
the minimum and maximum errors envelop the shaded region.
{\bf [{Left}]} Approximation error as a function of the number of epochs
elapsed.
{\bf [{Center}]} Approximation error as a function of flops.
{\bf [{Right}]} Approximation error as a function of CPU time.
See Section~\ref{sec:experiments-dense} for details.}
\label{fig:dense}
\end{center}
\end{figure}

\subsubsection{Unstructured Submatrices} \label{sec:experiments-dense}

Next, we consider some examples where we cannot exploit the structure of the submatrices
to accelerate the block Kaczmarz method.  Although our theory does not yield
higher rates of convergence, the numerical evidence still suggests that the block
Kaczmarz method offers a decisive improvement over the simple Kaczmarz method.

First, we consider a $300 \times 100$ real matrix $\mtx{A}_{\rm rdm}$ whose rows are drawn
independently and uniformly at random from the $\ell_2$ unit sphere.  We partition
the rows of this matrix into $10$ equal blocks of $30$ consecutive rows each.
A heuristic application of the Bai--Yin Law~\cite[Thm.~2]{BY93:Limit-Smallest}
shows that the singular values of each $30 \times 100$ submatrix fall in the
range $1 \pm \sqrt{30/100}$.  Thus, the paving parameters $\alpha \approx 0.2$
and $\beta \approx 2.4$, and the condition number of each block is bounded by
$\beta / \alpha \approx 4.8$.

We perform an experiment with this matrix that follows the same methodology as
in Section~\ref{sec:experiments-dense}.
To implement the block Kaczmarz update rule~\eqref{eqn:block-kacz}, we use
the {\tt CGLS} algorithm to solve the underdetermined least-squares problem.
We use warm starts, and we halt the least-squares iteration before it has
fully converged to control the computational cost.
Our code counts the actual number of (real) flops during
each iteration, and it measures the actual CPU time that is expended.

Figure~\ref{fig:dense} shows three different views of the data we collected
during this experiment.  The left panel shows the approximation error for
the simple and block Kaczmarz method as a function of the number of epochs
elapsed.  As we expect from the discussion in Section~\ref{sec:unstruct},
the two algorithms show a similar rate of convergence in this view.  In
fact, the block method does somewhat better than the theory predicts.
In the center panel, we chart the approximation error as a function of the
number of flops performed.  The simple algorithm requires about
$5 \times 10^6$ flops to achieve an error of $10^{-11}$, while the block
algorithm takes $2 \times 10^7$ flops to reach the same point.
Thus, the block method needs four times as much arithmetic.  But the story
is not over.  The right panel displays the approximation error as a function
of CPU time.  We discover that our implementation of the block method is 10
times faster than the simple method!  It is dangerous to draw conclusions
about the efficiency of an algorithm from the behavior of a {\sc Matlab} script.
Nevertheless, we regard this experiment as limited evidence
of the computational advantages of the block method that we outlined in Section~\ref{sec:why}.

\begin{wrapfigure}{r}{0.5\textwidth}
\begin{center}
\includegraphics[width=.5\textwidth]{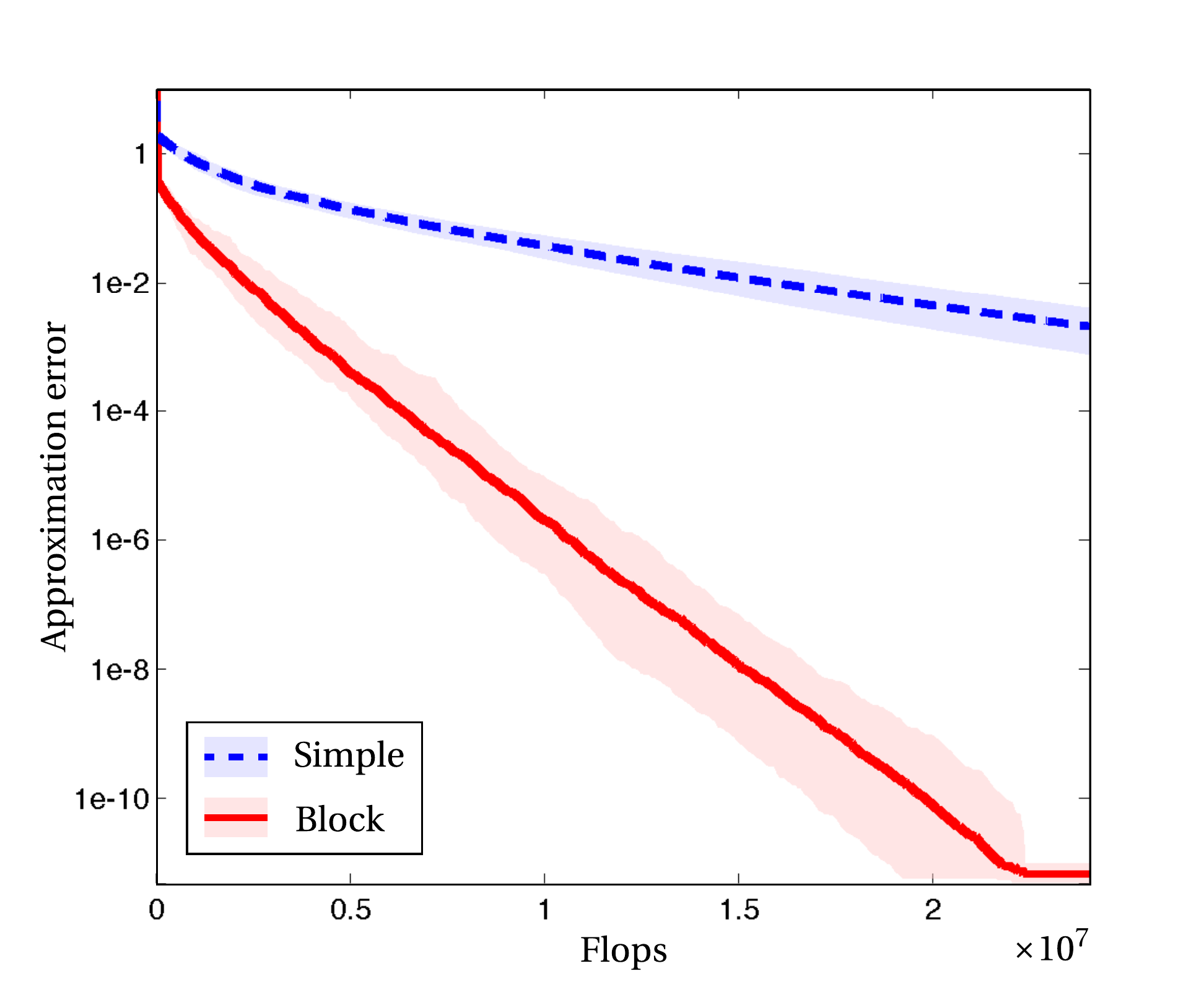}
\caption{({Convergence with a Coherent Matrix}){\bf .}
The matrix $\mtx{A}_{\rm coh}$ is a fixed $300 \times 100$ matrix with strongly correlated
rows that are partitioned into 10 blocks of equal size.
We compare the approximation error $\enorm{\vct{x}_j - \vct{x}_{\star}}$ for
the iterates generated by each algorithm as a function of the number of flops
performed.  The heavy line denotes the median error over 100 trials;
the minimum and maximum error describe the boundaries of the shaded region.
See Section~\ref{sec:experiments-dense} for details.}
\label{fig:coh}
\end{center}
\end{wrapfigure}

Finally, we consider one other type of unstructured matrix, where our analysis
offers little guidance.  Let $\tilde{\mtx{A}}_{\rm coh}$ be a $300 \times 100$ matrix whose
entries are drawn independently at random from the uniform distribution on $[0.5, 1]$.
We form a standardized matrix $\mtx{A}_{\rm coh}$ by normalizing the rows of
$\tilde{\mtx{A}}_{\rm coh}$.
Since this matrix is dense and positive, the rows of this matrix tend to be strongly
correlated with each other; the maximum inner product between rows is typically $0.98$
or higher.  Furthermore, the norm of this matrix is usually close to the maximum
possible value $\sqrt{300}$.
  Our theory provides no map for this
\emph{terra incognita}.
Yet we brazenly partition the rows into $10$ equal blocks, each with $30$ contiguous rows,
as in the previous example.

We perform the same experiment for $\mtx{A}_{\rm coh}$ as we did with
$\mtx{A}_{\rm rdm}$ to compare the behavior of the simple
Kaczmarz method and the randomized Kaczmarz method.
Figure~\ref{fig:coh} shows the results of this trial.
We see that the simple Kaczmarz method scarcely reduces the error at all,
while the block method achieves a healthy rate of convergence.  
The paper~\cite{NW12:Two-Subspace-Projection} provides an
analysis of this example in the case where each block contains
two rows, but we do not yet have a complete explanation for the performance
of Algorithm~\ref{alg:randomized-block-kaczmarz} when the blocks are larger.


\section{Related Work and Future Directions} \label{sec:future}

We conclude with a short discussion about recent research on randomized Kaczmarz methods and block variants of the simple Kaczmarz method.  Finally, we offer some musings about opportunities for further research.

\subsection{Kaczmarz Methods with Randomized Control}

The Kaczmarz method was originally introduced in the paper~\cite{Kac37:Angenaeherte-Aufloesung}.  It was reinvented by researchers in tomography~\cite{GBH70:Algebraic-Reconstruction} under the appellation ``algebraic reconstruction technique'' (ART).  See Byrne's book~\cite{Byr08:Applied-Iterative} for a contemporary summary of this literature.

The classical variants of the Kaczmarz method rely on deterministic mechanisms for selecting a row at each iteration.  Indeed, the simplest version just cycles through the rows in order.  It has long been known that the cyclic control scheme performs badly when the rows are arranged in an unhappy order~\cite{HS78:Angles-Null}.
The literature contains empirical evidence~\cite{HM93:Algebraic-Reconstruction} that \emph{randomized} control mechanisms may be more effective, but until recently there was no compelling theoretical analysis to support this observation.

The paper~\cite{SV09:Randomized-Kaczmarz} of Strohmer and Vershynin is significant because it provides the first explicit convergence proof for a randomized variant of the Kaczmarz algorithm.  This work establishes that a randomized control scheme leads to an expected linear convergence rate, which can be written in terms of geometric properties of the matrix.  In contrast, deterministic convergence analyses that appear in the literature often lead to expressions, e.g.,~\cite[Eqn.~(1.2)]{XZ02:Method-Alternating}, whose geometric meaning is not evident.



In the wake of Strohmer and Vershynin's work~\cite{SV09:Randomized-Kaczmarz}, several other researchers have written about randomized versions of the Kaczmarz scheme and related topics.  In particular,
Needell demonstrates that the randomized Kaczmarz method converges, even when the linear system is inconsistent~\cite{Nee10:Randomized-Kaczmarz}.
Zouzias and Freris~\cite{ZF12:Randomized-Extended} exhibit a randomized procedure, based on ideas from~\cite{Pop98:Extensions-Block-Projections}, that can reduce the size of the residual $\vct{e}$.  Leventhal and Lewis~\cite{LL10:Randomized-Methods} provide an analysis of a randomized iteration for solving least-squares problems with polyhedral constraints, while Richt{\'a}rik and Tak{\'a}{\v c} have extended these ideas to more general optimization problems~\cite{RT11:Iteration-Complexity}.  Some other references include~\cite{EN11:Acceleration-Randomized,RT11:Iteration-Complexity,CP12:Almost-Sure-Convergence,NW12:Two-Subspace-Projection}.


\subsection{Block Kaczmarz Methods}

The block Kaczmarz update rule~\eqref{eqn:block-kacz} we are studying is originally due to Elfving~\cite[Eqn.~(2.2)]{Elf80:Block-Iterative-Methods}.  This update is a special case of a general framework due to Eggermont et al.~\cite{EHL81:Iterative-Algorithms}.  Byrne describes a number of other block Kaczmarz methods in his book~\cite[Chap.~9]{Byr08:Applied-Iterative}.

By now, there is an extensive literature on the convergence behavior of block projection methods, a class of algorithms that includes block Kaczmarz schemes.  In particular, we call out the work of
Xu and Zikatanov~\cite{XZ02:Method-Alternating}, which contains a refined convergence analysis that applies to a wide range of algorithms.  Nevertheless, to our knowledge, Algorithm~\ref{alg:randomized-block-kaczmarz} is the only block Kaczmarz method that offers an (expected) linear rate of convergence that depends explicitly on geometric properties of the system matrix $\mtx{A}$ and its submatrices $\mtx{A}_{\tau}$.

Most of the literature on block Kaczmarz methods assumes that a partition $T$ of the rows of the matrix is provided as part of the problem data.  We are aware of some research on the prospects for partitioning a matrix in a manner that is favorable for block Kaczmarz methods.  In particular, Popa~\cite{Pop99:Block-Projections-Algorithms} has introduced an algorithm for partitioning a sparse matrix so that each block contains mutually orthogonal rows.  Popa has pursued this idea in a sequence of papers, including~\cite{Pop01:Fast-Kaczmarz-Kovarik,Pop04:Kaczmarz-Kovarik-Algorithm}.  We believe that our work is the first to recognize the natural connection between the paving literature and the block Kaczmarz method.

\subsection{Some Future Directions}

There are a number of interesting open questions connected with Algorithm~\ref{alg:randomized-block-kaczmarz}.  First, empirical experiments make it clear that a control strategy based on sampling \emph{without} replacement is far more effective than our strategy based on sampling \emph{with} replacement.  At present, there is no compelling explanation for this phenomenon (but see~\cite{RR12:Beneath-Valley}).  Second, we think that other types of block Kaczmarz update rules~\cite[Chap.~9]{Byr08:Applied-Iterative} would also submit to a convergence analysis similar to the proof of Theorem~\ref{thm:convergence}.
Third, it might be worthwhile to extend the argument of Zouzias and Freris~\cite{ZF12:Randomized-Extended} to develop a method with a smaller convergence horizon.  

Block algorithms for numerical linear algebra are almost as old as the field itself.  Gauss himself suggested a block version of his algorithm for solving linear systems~\cite{Gau95:Theoria-Combinationis,Ben09:Key-Moments}.  Many other algorithms for numerical linear algebra~\cite{GVL96:Matrix-Computations} and least-squares problems~\cite{Bjo96:Numerical-Methods} admit natural block variants.  The field of optimization also contains a wide variety of block methods, such as~\cite{AC89:Block-Iterative-Projection,Tse93:Dual-Coordinate}.

We believe that row pavings can also play a role in the development and analysis of other types of block algorithms.  For instance, it is straightforward to develop a block version of the randomized Gauss--Seidel algorithm introduced in~\cite{LL10:Randomized-Methods}.   Using the techniques in this paper, we can easily bound the rate of convergence for this algorithm in terms of the properties of a \emph{column paving} of the system matrix.  It would be interesting to pursue this example and others.


\appendix

\section{The Fast Incoherence Transform} \label{app:fit}

The goal of this Appendix is to prove Proposition~\ref{prop:fit}, which states that the fast incoherence transform makes a standardized matrix incoherent with high probability.  To that end, suppose that $\mtx{A}$ is a matrix with $n$ unit-norm rows, denoted $\vct{a}_1, \dots, \vct{a}_n$.  Recall that the fast incoherence transform is the matrix $\mtx{S} = \mathbf{F}\mtx{E}$, where $\mathbf{F}$ is the $n \times n$ unitary DFT and $\mtx{E}$ is a diagonal matrix whose entries $\xi_1, \dots, \xi_n$ are independent Rademacher random variables.

Consider the matrix $\mtx{W} := \mtx{SA}$, and introduce the Gram matrix of its rows $\mtx{G} := \mtx{WW}^* = \mtx{SAA}^* \mtx{S}^*$. Expanding this product, we find that the $(i, \ell)$ entry of $\mtx{G}$ takes the form
$$
g_{i\ell} = \sum\nolimits_{jk} \xi_j \xi_k \cdot
	{\rm f}_{ij} \bar{\rm f}_{\ell k} \cdot \ip{ \vct{a}_j }{ \vct{a}_k },
$$
where we write ${\rm f}_{ij}$ for the $(i, j)$ entry of the DFT matrix.  (We use the convention that the inner product is antilinear in the second coordinate.)  This expression allows us to calculate the expectation of the Gram matrix with ease.  Since the rows $\mathbf{f}_1, \dots, \mathbf{f}_n$ of the DFT form an orthonormal family,
$$
\Expect g_{i\ell} = \sum\nolimits_j {\rm f}_{ij} \bar{\rm f}_{\ell k} \enormsq{\vct{a}_j}
	= \ip{ \mathbf{f}_{i} }{ \mathbf{f}_\ell } = \delta_{i\ell},
$$
where $\delta_{i\ell}$ is the Kronecker delta.  Note that we have invoked the standardization assumption here.

The real content of the argument is to obtain a bound on \emph{how much} the entries of the Gram matrix deviate from their expected values.  Fix a pair $(i, \ell)$ of indices, not necessarily distinct, and define the random variable
$$
Y := \abs{ g_{i\ell} - \delta_{i\ell} }.
$$
We can rewrite $Y$ is a more symmetric manner:
\begin{equation} \label{eqn:Y-chaos}
Y = \abs{ \sum\nolimits_{j \neq k} \xi_j \xi_k \cdot y_{jk} }
\quad\text{where}\quad
y_{jk} := \frac{1}{2} \left( {\rm f}_{ij} \bar{\rm f}_{\ell k}
	+ \bar{\rm f}_{ik} {\rm f}_{\ell j}  \right) \ip{ \vct{a}_j }{ \vct{a}_k }.
\end{equation}
We see that the random variable $Y$ is a symmetric, homogeneous, second-order Rademacher chaos.  We can bound the probability that $Y$ is large by invoking a result of Hanson and Wright~\cite{HW71:Bound-Tail}; see~\cite[Chap.~8]{FR12:Mathematical-Introduction} for a modern proof.

\begin{proposition}[Hanson--Wright] \label{prop:hw}
Consider the chaos variable $Y$ defined in~\eqref{eqn:Y-chaos}.  Let $\mtx{Y}$ be the Hermitian matrix whose $(j, k)$ entry is $y_{jk}$ and whose diagonal entries are zero.  Then
$$
\Prob{ Y \geq t }
	\leq 2 \exp\left\{ - \cnst{c_{hw}} \cdot \min\left\{ \frac{t}{\norm{\mtx{Y}}}, \frac{t^2}{\fnormsq{\mtx{Y}}} \right\} \right\}
\quad\text{for $t \geq 0$.}
$$
The number $\cnst{c_{hw}}$ is a positive, universal constant.
\end{proposition}

To apply this result, we need to obtain bounds for the Frobenius norm and spectral norm of the matrix $\mtx{Y}$.  Note that
$$
\mtx{Y} = \frac{1}{2} (\mtx{Z} + \mtx{Z}^*)
\quad\text{where}\quad
\mtx{Z} := \diag(\mathbf{f}_i) \cdot ( \mtx{AA}^* - \Id ) \cdot \diag(\mathbf{f}_\ell)^*,
$$
and $\diag(\cdot)$ converts a vector into a diagonal matrix.  We may now calculate the required norms of $\mtx{Y}$.  First,
\begin{equation} \label{eqn:Y-specnorm}
\norm{\mtx{Y}} \leq \norm{\mtx{Z}}
	\leq \norm{\diag(\mathbf{f}_i)} \cdot \norm{\mtx{AA}^* - \Id} \cdot \norm{\diag( \mathbf{f}_\ell )}
	= \frac{1}{n} \norm{\mtx{AA}^* - \Id}
	\leq \frac{1}{n} \normsq{\mtx{A}}.
\end{equation}
The first inequality depends on the convexity of the spectral norm and its invariance under the conjugate transpose.  The third relation holds because the entries of the vectors $\mathbf{f}_i$ and $\mathbf{f}_\ell$ all have magnitude $n^{-1/2}$.  The last inequality follows because $\norm{\mtx{AA}^* - \Id} \leq \max\{ \normsq{\mtx{A}} - 1, 1 \} \leq \normsq{\mtx{A}}$.  For similar reasons,
\begin{equation} \label{eqn:Y-frobnorm}
\fnormsq{\mtx{Y}} \leq \fnormsq{\mtx{Z}}
	= \frac{1}{n^2} \fnormsq{\mtx{AA}^* - \Id}
	< \frac{1}{n^2} \fnormsq{\mtx{AA}^*}
	\leq \frac{1}{n} \normsq{\mtx{A}}.
\end{equation}
The strict inequality holds because $\mtx{AA}^*$ has a unit diagonal, which the identity matrix cancels off.  The final bound follows from the interpolation $\fnormsq{\mtx{P}} \leq \trace(\mtx{P}) \norm{\mtx{P}}$, valid when $\mtx{P}$ is positive semidefinite.

Next, let us instate the hypothesis~\eqref{eqn:fit-hyp} from Proposition~\ref{prop:fit}:
$$
\normsq{\mtx{A}} \leq \frac{\cnst{c_{fit}} \cdot n}{\log^3(1+n)}.
$$
Introduce this assumption into the bounds~\eqref{eqn:Y-specnorm} and~\eqref{eqn:Y-frobnorm}, and apply the Hanson--Wright inequality, Proposition~\ref{prop:hw}, to reach
$$
\Prob{ Y \geq t } \leq 2 \exp\left\{ -\frac{\cnst{c_{hw}}}{\cnst{c_{fit}}} \cdot \log^3(1+n)
	\cdot \min\{ t, t^2 \} \right\}.
$$
We may set the constant $\cnst{c_{fit}} := \cnst{c_{hw}} \cnst{c_{inc}}^2 / 3$.  For the choice $t = \cnst{c_{inc}} / \log(1+n)$, we discover that
$$
\Prob{ Y \geq \frac{\cnst{c_{inc}}}{\log(1+n)} } \leq 2 (1 + n)^{-3}.
$$
In other words, it is unlikely that any single pair of rows from $\mtx{W} = \mtx{SA}$ has an inner product much different from its expectation.

To complete the argument, we unfix the pair $(i, \ell)$ of indices.  Forming a union bound over all $n(n+1)/2$ choices where $i \leq \ell$, we conclude that
$$
\Prob{ \max_{i,\ell} \abs{g_{i\ell} - \delta_{i\ell}} \geq \frac{\cnst{c_{inc}}}{\log(1+n)}}
	\leq (1+n)^{-1}.
$$
Therefore, with high probability $\mtx{W}$ is an incoherent matrix that is nearly standardized.

\section*{Acknowledgments}

We would like to thank Michael Mahoney, Ben Recht, Thomas Strohmer, Steve Wright for helpful discussions about randomized linear algebra and numerical experiments.  Roman Vershynin provided insight on the random paving literature.  Michael McCoy explained advanced plotting techniques in {\sc Matlab}, and Margot Stokol shared her expertise on color theory.  JAT was supported in part by ONR awards N00014-08-1-0883 and N00014-11-1002, AFOSR award FA9550-09-1-0643, DARPA award N66001-08-1-2065, and a Sloan Research Fellowship.


\bibliography{rk-mathscinet}
\bibliographystyle{myalpha}


\end{document}